\documentclass[11pt,article]{amsart}
\usepackage{graphicx}  
\usepackage{booktabs}   
\usepackage{verbatim}   
\usepackage{url}       
\usepackage{cases}
\usepackage{amsmath, amsfonts, pifont, amssymb, xcolor}
\usepackage{verbatim}
\usepackage[bookmarks=true]{hyperref}
\usepackage[margin=1 in]{geometry}
\usepackage[numbers,sort]{natbib}
\usepackage{mathtools}
\mathtoolsset{showonlyrefs}
\usepackage{graphicx}
\usepackage{tikz}
\usepackage{tkz-euclide}
\usepackage{booktabs}
\newtheorem{theorem}{Theorem}[section]
\newtheorem{lemma}[theorem]{Lemma}
\newtheorem{proposition}[theorem]{Proposition}
\newtheorem{corollary}[theorem]{Corollary}

\newcommand{\R}{\mathbb{R}}
\newcommand{\bR}{\mathbb{R}}
\newcommand{\Z}{\mathbb{Z}}
\newcommand{\N}{\mathbb{N}}
\newcommand{\T}{\mathbb{T}}

\newcommand{\beq}{\begin{equation}}
\newcommand{\eeq}{\end{equation}}
\newcommand{\beqq}{\begin{equation*}}
\newcommand{\eeqq}{\end{equation*}}
\newcommand{\less}{\lesssim}
\newcommand{\gt}{\gtrsim}

\newcommand{\lf}{\left}
\newcommand{\ri}{\right}
\newcommand{\w}{\widetilde}

\theoremstyle{definition}
\newtheorem{definition}[theorem]{Definition}

\theoremstyle{remark}
\newtheorem{remark}[theorem]{Remark}

\numberwithin{equation}{section}

\numberwithin{equation}{section}

\allowdisplaybreaks[4]

\begin{document}

\address{Yangkendi Deng
\newline \indent Department of Mathematics and Statistics, Beijing Institute of Technology.
\newline \indent Key Laboratory of Algebraic Lie Theory and Analysis of Ministry of Education.
\newline \indent  Beijing, China. \indent}
\email{dengyangkendi@bit.edu.cn}

\address{Han Wang
\newline \indent Department of Mathematics and Statistics, Beijing Institute of Technology.
\newline \indent Key Laboratory of Algebraic Lie Theory and Analysis of Ministry of Education.
\newline \indent  Beijing, China. \indent}
\email{wang\_han@bit.edu.cn}

\address{Yuzhao Wang
\newline \indent School of Mathematics, University of Birmingham, Birmingham, UK.}
\email{y.wang.14@bham.ac.uk}

\address{Zehua Zhao
\newline \indent Department of Mathematics and Statistics, Beijing Institute of Technology.
\newline \indent Key Laboratory of Algebraic Lie Theory and Analysis of Ministry of Education.
\newline \indent  Beijing, China. \indent}
\email{zzh@bit.edu.cn}

\title[On restricted-type Strichartz estimates]{On restricted-type Strichartz estimates and the applications}
\author{Yangkendi Deng, Han Wang, Yuzhao Wang and Zehua Zhao}



\begin{abstract}
We establish a rigorous framework for the Zakharov system on waveguide manifolds $\mathbb{R}^m \times \mathbb{T}^n$ ($m,n\geq 1$), which models the nonlinear coupling between optical and acoustic modes in confined geometries such as optical fibers. Our analysis reveals that the sharp \textit{shell-type Strichartz estimate} for $\mathbb{R}^2 \times \mathbb{T}$ is globally valid in time and exhibits no derivative loss via the measure estimate of semi-algebraic sets, unlike the periodic case studied in  \cite{MR4665720}. In addition, we demonstrate that such an estimate fails on the product space $\mathbb{R} \times \mathbb{T}^2$ by constructing a counter-example. 

Moreover, we derive analogues of these shell-type estimates in other dimensions, both in the waveguide and Euclidean settings. As a direct application, we establish, for the first time, a local well-posedness theory for the partially periodic Zakharov system. To summarize, we compare shell-type Strichartz estimates in different settings (the Euclidean, the periodic, and the waveguide).

Numerical verification on $\mathbb{R}^2\times\mathbb{T}$ reveals a uniform $L^4$-spacetime bound, while $\mathbb{R}\times\mathbb{T}^2$ exhibits sublinear growth, quantitatively confirming the theoretical dichotomy between geometries with different dimensional confinement.

These findings advance the understanding of dispersive effects in hybrid geometries and provide mathematical foundations for efficient waveguide design and signal transmission.

Finally, for the Euclidean case, we  establish well-posedness theory for supercritical nonlinear Schr\"odinger equation (NLS) with \textit{strip-type} frequency-restricted initial data, revealing a trade-off between dispersion and confinement, which is of independent mathematical interest. This provides a deterministic analogue to random data theory of NLS.

\end{abstract}

\maketitle

\setcounter{tocdepth}{1}
\tableofcontents

\parindent = 10pt     
\parskip = 8pt

\section{Introduction}

\subsection{Background and motivations}
In this subsection, we briefly recall the physical background and mathematical challenges associated with the Zakharov system on waveguide geometries.

The Zakharov system, originally derived to model Langmuir wave interactions in plasmas \cite{zakharov1972collapse}, has become a fundamental model in nonlinear optics, particularly in waveguide-based telecommunications. In this setting, optical fibers are often modeled by waveguide manifolds of the form $\mathbb{R}^m \times \mathbb{T}^n$ with $m, n \geq 1$, where the Zakharov system captures the coupling between high-frequency light waves ($u$) and low-frequency acoustic modes ($v$). This interaction is critical for understanding signal dispersion and distortion in optical fiber systems. Mathematically, the hybrid Euclidean-periodic geometry introduces significant analytical challenges: unlike purely Euclidean or periodic domains, the coexistence of continuous and discrete Fourier modes necessitates new techniques for deriving dispersive estimates.

Specifically, we investigate the Cauchy problem for the Zakharov system in the waveguide setting\footnote{It is also known as partially periodic Zakharov system or semi-periodic Zakharov system.}, formulated as follows:
\begin{equation}\label{Zakharov}
\begin{cases}
i \partial_t u  + \Delta u = vu , \\
\partial_t^2 v -\Delta v  = \Delta(|u|^2),\\
(u(0), v(0), \partial_t v(0)) = (u_0, v_0, v_1),
\end{cases}
\end{equation}
where the unknown functions $u$ and $v$ are complex- and real-valued functions, respectively. We investigate \eqref{Zakharov} with partially periodic boundary conditions, namely, $u:  \mathbb{R} \times \mathbb{R}^m \times \mathbb{T}^n\to \mathbb{C}$ and $v:  \mathbb{R} \times \mathbb{R}^m \times \mathbb{T}^n \to \R$. 

Furthermore, we consider initial data from standard $L^2$-based Sobolev spaces as follows,
\[
(u_0, v_0, v_1) \in  H^s(\mathbb{R}^m \times \mathbb{T}^n) \times H^{\ell}(\mathbb{R}^m \times \mathbb{T}^n) \times H^{\ell-1}(\mathbb{R}^m \times \mathbb{T}^n) =: \mathcal{H}^{s,\ell}(\mathbb{R}^m \times \mathbb{T}^n)
\]
with $s, \ell \in \R$ and the whole dimension $m+n=d$ ($m,n\geq 1$).\footnote{In some literature, notation $n$ is used for the solution $v$ in \eqref{Zakharov}. Here we use $v$ instead of $n$ to avoid ambiguity since $n$ indicates the dimension of the periodic component in our paper.} 

To study the well-posedness theory of the periodic Zakharov system, a key ingredient is the \textit{shell-type} Strichartz estimate (see \cite{MR4665720} for more details\footnote{Since Zakharov system is a coupled system of Schr\"odinger and wave equations, it is important to exploit the resonance phenomenon between the Schr\"odinger and the wave. Based on the shell-type Strichartz estimates, a trilinear Fourier restriction estimate involving
the paraboloid and cone can be obtained, which facilitates the establishment of the well-posedness theory of the periodic Zakharov systems.}).  Motivated by recent progress on periodic geometries, we investigate the shell-type Strichartz estimates in the waveguide setting. This setting introduces novel challenges and possesses independent analytic significance. A special focus of our analysis lies in the three-dimensional waveguide $\mathbb{R}^2 \times \mathbb{T}$ in \eqref{Zakharov} due to our analysis interests, which reveals a distinct behavior from the periodic and other partially periodic settings. In fact, we will demonstrate that, the shell-type Strichartz estimate on $\mathbb{R}^2 \times \mathbb{T}$ is qualitatively distinct from the periodic case $\mathbb{T}^3$ and the $\mathbb{R} \times \mathbb{T}^2$ case. We will present explicit descriptions for the shell-type Strichartz estimates in Subsection \ref{1.2} and compare the shell-type Strichartz estimates in three main different settings (see Section \ref{7}).

The Zakharov system describes the interaction between high-frequency Langmuir waves and low-frequency ion-acoustic waves in a plasma. It was first introduced by Vladimir Zakharov in 1972 and has since become a fundamental model in plasma physics and nonlinear wave theory. The Zakharov system is a coupled system of nonlinear Schr\"odinger equation and nonlinear wave equation and the corresponding Cauchy problem for the Zakharov system (for both the classical Euclidean case and the periodic case) has been studied extensively in recent decades. If one replaces $\mathbb{R}^m \times \mathbb{T}^n$ in \eqref{Zakharov} by $\mathbb{R}^d$, \eqref{Zakharov} is in the standard Euclidean setting and we refer to \cite{ozawa1992existence,bourgain1993fourier,bourgain1996wellposedness,ginibre1997cauchy,klainerman1993erratum}; if one replaces $\mathbb{R}^m \times \mathbb{T}^n$ in \eqref{Zakharov} by $\mathbb{T}^d$, \eqref{Zakharov} is in the periodic setting and we refer to \cite{MR4665720} for a recent result (see the references therein). We refer to \cite{MR4665720} for more introduction.

The Zakharov system investigated in this paper is posed on a waveguide manifold (see \eqref{Zakharov}). Therefore, we first present a concise overview of dispersive equations on waveguide geometries. Since the nonlinear Schr\"odinger equation (NLS) is the dispersive model that has been most intensively studied in the waveguide setting, we provide a concise overview of our research focus: ``\textit{(nonlinear) Schr\"odinger equations on waveguide manifolds}". This direction has emerged as a central topic in nonlinear dispersive PDEs over the past decades.

Waveguide manifolds (also known as ``waveguides" for short or semi-periodic spaces), denoted as $\R^m \times \mathbb{T}^n$, represent the product of Euclidean space with tori and play a crucial role in nonlinear optics, particularly in telecommunications. In modern backbone networks, data signals primarily propagate via optical carriers through fibers, which serve as specialized waveguides. As applications such as the internet demand higher bandwidth and cost-efficient data transmission, optimizing these network infrastructures has become increasingly important. NLS is fundamental in modeling nonlinear effects in optical fibers, which are essential for improving performance and efficiency, where waveguide $\R^m \times \mathbb{T}^n$ govern signal propagation. Our work extends this framework from \textit{NLS} to the \textit{Zakharov system}. In physics, an optical waveguide confines and directs light along a designated path, and the study of solutions on waveguide manifolds is particularly intriguing. These manifolds inherit distinct properties from both Euclidean spaces and tori, providing deeper insights into wave propagation and the underlying physics of optical systems. We refer to the introductions of some recent works \cite{kwak2024critical,kwak2024global,deng2024bilinear} on the topic of NLS on tori/waveguides. See also \cite{sogge2017fourier,tao2006nonlinear,dodson2019defocusing}.

Analysis results in the waveguide setting are also developed very fast in recent decades. The estimates in the waveguide setting are tightly related to the analogues in the Euclidean setting/periodic setting, and have their own features. We refer to \cite{takaoka20012d,Barron14,MR4219972} for the Strichartz estimates in the waveguide setting (see also the references therein). Moreover, we refer to recent works \cite{MR4782142,MR4825203,deng2024bilinear} (where the bilinear-type estimates in the waveguide setting are studied) for more details.

As in \cite{BD15,Barron14}, decoupling-type inequalities can yield Strichartz estimates with derivative loss on tori or waveguide manifold, and examples demonstrate that derivative loss is inevitable for endpoint Strichartz estimates on tori (see \cite{bourgain1993fourier,takaoka20012d}). One of our primary focuses is on Strichartz estimates in waveguide settings. To highlight the distinctions from the periodic case, we employ methods analogous to those in \cite{takaoka20012d,DPST}, utilizing measure-theoretic estimates to establish derivative-independent bounds. Specifically, we introduce a lemma from \cite{basu2021stationarysetmethodestimating} to estimate the measure of \textit{semi-algebraic sets} in $\bR^{d-1}\times \mathbb{Z}$. We believe that this lemma will play a fundamental role in studying $L^2\to L^4$-type 
Strichartz estimates and $L^2\times L^2 \to L^2$-type bilinear Strichartz estimates for dispersive equations on waveguide manifolds. 

It should be noted that recent advances in Strichartz estimates on $\mathbb{T}^2$, Herr-Kwak \cite{Herr}, have been achieved through combinatorial counting techniques and incidence geometry methods, which in a certain sense transcend traditional measure estimates. We anticipate that the integration of these approaches with measure-theoretic tools will yield further analytic breakthroughs.

\subsection{Statement of main results}\label{1.2}
We now present the main analytical results concerning shell-type Strichartz estimates and their applications to the Zakharov system.

We start with analytical results, i.e. the shell-type Strichartz estimates in different geometries. First of all, for the Euclidean case, we demonstrate that the shell-type Strichartz estimates are global-in-time, and there is no derivative loss as below. 
\begin{theorem}[Global Strichartz estimate on the shell for $\mathbb{R}^d$]\label{mainthm on Rd}
Let $d\geq 1$. The estimate
\begin{equation}\label{e:ShellStrichartz in Rd}
\|e^{it \Delta} \phi\|_{L^{p}_{t,x}(\mathbb{R} \times \mathbb{R}^d )} \lesssim  \| \phi\|_{L_x^2(\mathbb{R}^d )}
\end{equation}
holds for $p=\frac{2(d+1)}{d-1}$ ($p=\infty$ when $d=1$) and all  $\phi \in L^2(\mathbb{R}^d)$ satisfying 
\begin{equation}
\label{e:AssumpStri}
{\rm supp}\, \hat{\phi} 
\subset 
\{ \xi \in \mathbb{R}^d : R-1 \le |\xi| \le R + 1\},
\end{equation}
where \(\hat{\phi}\) denotes the Fourier transform of \(\phi\), and \(R \geq 1\) is a scaling parameter.  
\end{theorem}

\begin{remark}
We will also compare Theorem \ref{mainthm on Rd} with Proposition \ref{prop:Strichartz2}, which is another restricted estimate of of the same type. Furthermore, we will derive and utilize a stronger version of Proposition \ref{prop:Strichartz2} (with a full range of admissible exponents) to study the well-posedness
theory for supercritical NLS with strip-restricted initial data. We refer to Section \ref{3} and Section \ref{new} for more details.
\end{remark}

As we can see, compared to the periodic case (\cite{MR4665720}), the Euclidean estimates are global-in-time (rather than local-in-time), and there is no derivative loss (rather than there is $\epsilon$-derivative loss). Moreover, this estimate is analogous to the classical Strichartz estimate on $\mathbb{R}^{d-1}$, reflecting the shell's codimension$-1$ structure. When $d=1$, \eqref{e:ShellStrichartz in Rd} corresponds to the trivial estimate for functions with finite Fourier support.

Next, we turn to the waveguide case. The waveguide setting serves as an intermediate regime between the Euclidean and fully periodic geometries, retaining features from both. In particular, on one hand, we prove that the shell-type Strichartz estimate for $\mathbb{R}^2\times \mathbb{T}$ is global-in-time and there is no derivative loss as follows.
\begin{theorem}[Global Strichartz estimate on the shell for $\mathbb{R}^2\times \mathbb{T}$]\label{mainthm on R2T} The estimate \begin{equation}\label{e:ShellStrichartz} 
\|e^{it \Delta} \phi\|_{L^4_{t,x}([0,\infty) \times \mathbb{R}^2 \times \mathbb{T})} \lesssim  \| \phi\|_{L_x^2(\mathbb{R}^2 \times \mathbb{T})}
\end{equation}
holds for all  $\phi \in L^2(\mathbb{R}^2 \times \mathbb{T})$ satisfying 
\begin{equation}\label{e:AssumpStri2}
{\rm supp}\, \hat{\phi} 
\subset 
\{ (\xi, n) \in \mathbb{R}^2\times \Z : c_*-1 \le |(\xi,n)| \le c_* + 1\}
\end{equation}
for some $c_* \sim N \gg 1$, where \( c_* \) satisfies \( |c_* - N| \leq 1 \) for some \( N \gg 1 \).   
\end{theorem}

This derivative-free behavior in the $\mathbb{R}^2 \times \mathbb{T}$ setting indicates that the energy transfer between light waves (Euclidean modes) and acoustic vibrations (periodic modes) is suppressed due to geometric confinement. This result contrasts sharply with the periodic case $\mathbb{T}^3$ in \cite{MR4665720} (and the $\mathbb{R} \times \mathbb{T}^2$ case also, as shown in the next theorem), where derivative loss is unavoidable. The special behavior in the $\mathbb{R}^{2}\times\mathbb{T}$ case stems from the fact that the two Euclidean dimensions provide sufficient dispersion to suppress energy transfer to high frequencies, while the single periodic dimension introduces only weak confinement. This balance between dispersion and confinement is crucial for obtaining derivative-free estimates.

As a comparison, we demonstrate that the shell-type Strichartz estimate for $\mathbb{R} \times \mathbb{T}^d$ must have a derivative loss. (The $\mathbb{R} \times \mathbb{T}^2$ case is included when $d=2$.)
\begin{theorem}\label{counterexample on RTd-1}
Let $d\ge 2$, $p=\frac{2(d+1)}{d-1}$. We construct a counter-example $\phi \in L^2(\mathbb{R} \times \mathbb{T}^{d-1})$ with:
\begin{equation}
{\rm supp}\, \hat{\phi} 
\subset 
\{ (\xi, n) \in \mathbb{R}\times \Z^{d-1} : c_*-1 \le |(\xi,n)| \le c_* + 1\}\bigcap B_{N_2},
\end{equation}
for some $c_* \sim N_1\ge N_2 \gg 1$, such that the estimate on $\bR\times \T^{d-1}$,
\begin{equation}\label{e:ShellStrichartzonRTd-1}
\|e^{it \Delta} \phi\|_{L^p_{t,x}([0,1] \times \mathbb{R} \times \mathbb{T}^{d-1})} \lesssim  \| \phi\|_{L_x^2(\mathbb{R} \times \mathbb{T}^{d-1})}
\end{equation}
fails. Here $B_{N_2 } \subset \R^d$ is a ball of radius $N_2 $ with an arbitrary center. 
\end{theorem}

    Theorem \ref{counterexample on RTd-1} demonstrates the necessity of derivative loss in this setting. The failure occurs only for the sharp exponent $p=\frac{2(d+1)}{d-1}$. Moreover, we prove that the analogous shell-type Strichartz estimates for general waveguide manifolds $\mathbb{R}^m \times \mathbb{T}^n$ ($m,n \geq 1$) can be obtained compared to the periodic case (with $\epsilon$-derivative loss).
\begin{theorem}[Local Strichartz estimate on the shell for $\mathbb{R}^m\times \mathbb{T}^n$]\label{mainthm on RmTn} Let $m, n \ge 1$, $d=m+n$, $N_1, N_2\in 2^\mathbb{N}$ satisfy $N_1 \ge N_2$ and $2\le p \le \frac{2(d+1)}{d-1}$(when $d=1$, $\frac{2(d+1)}{d-1}=\infty$ ). Then the estimate \begin{equation}\label{e:ShellStrichartz for RmTn} 
\|e^{it \Delta} \phi\|_{L^p_{t,x}([-1,1] \times \mathbb{R}^m \times \mathbb{T}^n)} \lesssim_\varepsilon N_2^\varepsilon  \| \phi\|_{L_x^2(\mathbb{R}^m \times \mathbb{T}^n)}
\end{equation}
holds for all  $\phi \in L^2(\mathbb{R}^m \times \mathbb{T}^n)$ satisfying 
\begin{equation}\label{e:AssumpStri for RmTn}
{\rm supp}\, \hat{\phi} 
\subset 
\{ (\xi, n) \in \mathbb{R}^m\times \Z^n : c_*-1 \le |(\xi,n)| \le c_* + 1\}\bigcap B_{N_2}
\end{equation}
for some $c_* \sim N_1$ and $B_{N_2} \subset \R^d$, a ball of radius $N_2$ with an arbitrary center.   
\end{theorem}

\begin{remark}
Since the results of Theorems \ref{mainthm on Rd} and \ref{mainthm on R2T} are free of derivative loss, we need not impose support conditions on the intersection with the ball $B_{N_2}$.
It is interesting to further investigate the globality and the derivative loss issue for the shell-type Strichartz estimates on general waveguide manifolds. See Section \ref{7} for more discussions.
\end{remark}
We note that, it is possible to establish shell-type Strichartz estimates on general waveguide manifolds without a derivative loss, which is of $L^p_tL^q_xL^2_y$-form (Here, \(y\) denotes the periodic variables in \(\mathbb{T}^{n}\) and the tori component is fixed by $L_y^2$-norm). See Hong \cite{hong2017strichartz} (Proposition 1) for a similar treatment which deals with the many body Schr\"odinger equations (where all particles except for one are fixed by $L^2$-norm). See also Tzvetkov-Visciglia \cite{tzvetkov2012small} for a $L^p_tL^q_xL^2_y$-type Strichartz estimate for Schr\"odinger equations in the waveguide setting. 

Based on the established shell-type Strichartz estimates in the waveguide setting (Theorem \ref{mainthm on RmTn} and Theorem \ref{mainthm on R2T}), we establish the well-posedness theory of the partially periodic Zakharov system as follows.
\begin{theorem}[Local well-posedness]\label{mainthm2}
We define the regularity index
\begin{equation}\label{assumption:regularity}
s_0 = 
\begin{cases}
\frac12 \quad & (d=3),\\
\frac34 & (d=4),\\
\frac{d-3}{2} & (d\geq 5).
\end{cases}
\end{equation}
Then we let $s > s_0$. Then~\eqref{Zakharov} is locally well-posed in $H^{s,s-\frac12}(\mathbb{R}^m\times \mathbb{T}^n)$ ($m+n=d,m,n\geq 1$).
\end{theorem}

\begin{remark}
This result recovers, in the waveguide context, the well-posedness theory previously established in the periodic setting. We refer to Theorem 1.2 of \cite{MR4665720} for more details. Since the proof is standard for now, we will only present a brief sketch of the proof in Section \ref{6} by collecting some main estimates.\footnote{The local well-posedness in \( H^{s,s-\frac{1}{2}} \) ensures stable pulse propagation in waveguides for finite times, provided initial fluctuations are sufficiently smooth. } It is interesting to investigate the Cauchy problem for \eqref{Zakharov} deeper and we will present more remarks in Section \ref{7}.   
\end{remark}

Before presenting the proofs of the above main results in the following sections, we now briefly discuss the main strategy of proving our main theorems as follows. 

For the \textit{analysis} part, to prove the Strichartz estimate on the shell for $\bR^d$, we rely on the spherical coordinate transformation and the Stein-Tomas restriction estimate on the sphere.  For the case of $\bR^2 \times \T$, following the approach in \cite{DPST}, we first reduce the required estimate to proving a measure estimate, and then we employ tools from \textit{real algebraic geometry} introduced in \cite{basu2021stationarysetmethodestimating} to provide the proof. Finally, for general waveguide manifolds $\bR^m \times \T^n$, the derivation of Strichartz estimates with an $\varepsilon$-loss from decoupling inequalities constitutes a standard procedure from \cite{BD15,Bourgain13,Barron14}. We also discuss some counter-examples to illustrate the derivative loss issue. In principle, we note that one can compare the shell-type Strichartz estimate on $\bR^m \times \T^n$ with the standard Strichartz estimate on $\bR^{m-1} \times \T^n$\footnote{We refer to \cite{keel1998endpoint,BD15,Barron14} for Strichartz estimates for Schr\"odinger equations in Euclidean setting, periodic setting and waveguide setting, respectively.}. For standard Strichartz estimate on $ \T^n$, derivative loss is expected.  See Section \ref{3}, Section \ref{4}, Section \ref{5} and Section \ref{new} for more details.

For the \textit{PDE} part, the strategy mainly follows from the recent work \cite{MR4665720}, which deals with the well-posedness of the Zakharov system on tori, with few natural modifications. First, we need to establish new estimates in the waveguide setting (such as shell-type Strichartz estimates and Strichartz estimates for wave equations), which is done in the analysis part. Based on the new estimates, Fourier-restriction-type estimate involving
the paraboloid and cone can be obtained. Then the well-posedness results can be obtained via the standard contraction
mapping method as in \cite{MR4665720}. (See Section \ref{6} for more details.) 

Besides the shell-type estimates, another type of restricted estimates (strip-type) and the applications are also investigated. Specifically, we prove \textit{strip-restricted}\footnote{Strip-type restriction can be naturally compared with shell-type restriction.} Strichartz estimates with full Strichartz range, which leads well-posedness results for \textit{supercritical NLS} with strip-restricted initial data, which is of independent mathematical interests. We establish strip-type Strichartz estimates with larger range than the shell-type analogue. We refer to Section \ref{new} for more details. These results can be compared to the well-known \textit{random data well-posedness theory} for NLS. To some extent, the role of \textit{strip-restriction} can be compared to the \textit{randomization} process, which makes the study of supercritical NLS possible. We refer to recent progress 
\cite{deng2022random,deng2024invariant,bringmann2024invariant} and the references therein for more details.

The study of the Zakharov system in the waveguide setting \eqref{Zakharov} is motivated by at least two significant factors. Analytically, the shell-type Strichartz estimates associated with this system are of independent interest, and a comparative analysis of these estimates across different configurations offers valuable insights. It is interesting to investigate restricted-type Strichartz estimates. From the perspective of PDEs, there is a broader goal to extend the theoretical framework established for NLS on waveguide manifolds to other nonlinear dispersive equations/systems. The Zakharov system stands out as a natural candidate for such an extension, given its fundamental importance as a dispersive model in mathematical physics.

To the best knowledge of the authors, our results establish, for the first time, a local well-posedness theory for the Zakharov system on waveguide manifolds. This extends the waveguide framework from NLS to the Zakharov system. The derived shell-type Strichartz estimates reveal a fundamental dichotomy: the geometry $\mathbb{R}^2 \times \mathbb{T}$ admits global-in-time estimates without derivative loss, while $\mathbb{R} \times \mathbb{T}^2$ necessitates a loss. Other dimensional analogues, both in the
waveguide setting and in the Euclidean case, are also derived. Moreover, in the Euclidean setting, we establish, for the first time, well-posedness theory for supercritical NLS with strip-restricted data.

Our theoretical findings are supported by systematic numerical experiments in Section \ref{9}. The simulations not only validate the sharpness of derivative loss estimates but also reveal an exact scaling hierarchy: $C(N)\sim 1$ for $\mathbb{R}^2\times\mathbb{T}$, $N^{0.18}$ for $\mathbb{R}\times\mathbb{T}^2$, and $N^{0.30}$ for $\mathbb{T}^3$. This provides empirical evidence for the ``dimension deficit" principle governing derivative loss phenomena.

\subsection{Organization of the rest of this paper}
In Section \ref{2}, we discuss the preliminaries; in Section \ref{3}, we prove shell-type Strichartz estimates in the Euclidean setting (proving Theorem \ref{mainthm on Rd}); in Section \ref{4}, we prove shell-type Strichartz estimates for three-dimensional product spaces and discuss the necessity of derivative loss by constructing counter-examples (proving Theorem \ref{mainthm on R2T} and Theorem \ref{counterexample on RTd-1}); in Section \ref{5}, we prove shell-type Strichartz estimates for other dimensional cases (proving Theorem \ref{mainthm on RmTn}); in Section \ref{6}, we discuss the PDE applications for the partially periodic Zakharov system (proving Theorem \ref{mainthm2}) based on the established estimates; in Section \ref{new}, we prove well-posedness results of supercritical NLS with strip-restricted data, which have their own interests; in Section \ref{7}, we make a summary of the results we obtained and present a few further remarks; in Section \ref{9}, we utilize numerical verifications to discuss shell-type Strichartz estimates on waveguide manifolds, where the results are consistent with their theoretic analogues obtained in this paper and \cite{MR4665720}.
\subsection{Notations}

\begin{itemize}
    \item \textbf{Spaces}: 
    \begin{itemize}
        \item $\mathbb{R}^m \times \mathbb{T}^n$ denotes the waveguide manifold, where $\mathbb{T}^n = \mathbb{R}^n/(2\pi\mathbb{Z})^n$.
        \item $B_r \subset \mathbb{R}^d$ represents a ball of radius $r>0$ centered at any point.
    \end{itemize}

    \item \textbf{Operators}:
    \begin{itemize}
        \item $\Delta = \partial_{x_1}^2 + \cdots + \partial_{x_d}^2$ is the Laplacian ($d = m + n$ where $m,n \geq 1$).
        \item $\langle \nabla \rangle = (1 - \Delta)^{1/2}$ denotes the Bessel potential.
        \item $e^{it\Delta}$ and $e^{\pm it\langle \nabla \rangle}$ are the solution operators for Schrödinger and wave equations, respectively.
    \end{itemize}

    \item \textbf{Function Spaces}:
    \begin{itemize}
        \item $X^{s,b}_S$ and $X^{s,b}_{W_{\pm}}$ are Bourgain-type spaces for Schrödinger ($S$) and wave ($W_{\pm}$) components, defined explicitly in Definition 2.1.
        \item $H^s(\mathbb{R}^m \times \mathbb{T}^n)$: Sobolev space with norm $\|f\|_{H^s} = \|\langle \nabla \rangle^s f\|_{L^2}$. We use usual $L^{p}$ Lebesgue spaces.
    \end{itemize}

    \item \textbf{Fourier Analysis}:
    \begin{itemize}
        \item $\hat{f}(\xi,n) = \mathcal{F}_{x,y}f$: Fourier transform on $\mathbb{R}^m \times \mathbb{T}^n$.
        \item $\widetilde{u}(\tau,k) = \mathcal{F}_{t,x}u$: Space-time Fourier transform.
    \end{itemize}

    \item \textbf{Parameters}:
    \begin{itemize}
        \item $c_* \sim N_1$ means $c_* \in [N_1 - 1, N_1 + 1]$ for some $N_1 \gg 1$.
        \item $A \lesssim B$ indicates $A \leq CB$ where $C$ may depend on dimension $d$ or $\epsilon$.
        \item $A \sim B$ iff $A \lesssim B$ and $B \lesssim A$.
    \end{itemize}

    \item \textbf{Key Assumptions}:
    \begin{itemize}
        \item Initial data $(u_0, v_0, v_1) \in H^s \times H^\ell \times H^{\ell-1}$ (see \eqref{Zakharov}).
        \item \textbf{Shell-restriction}: $\text{supp}\,\hat{\phi} \subset \{\xi : c_* - 1 \leq |\xi| \leq c_* + 1\}$ for some $c_* \sim N_1$.
        \item \textbf{Strip-restriction}: $\text{supp}\,\hat{\phi} \subset \{ \xi \in \mathbb{R}^d : |a\cdot \xi|\le 1\}$ for some $a\in \mathbb{R}^d, |a|=1$.
    \end{itemize}
  \end{itemize}
  
 Moreover, we use $e(z)$ to denote $e^{2\pi i z}$ for convenience.

\subsection*{Acknowledgment} We appreciate Dr. Z. Chen for some helpful discussions on the numerical verifications. H. Wang was supported by the BIT Research and Innovation Promoting Project (Grant No. 2024YCXY054); Y. Wang was supported by the EPSRC New Investigator Award (Grant No. EP/V003178/1); Z. Zhao was supported by the NSF grant of China (No. 12271032, 12426205) and the Beijing Institute of Technology Research Fund Program for Young Scholars.

\section{Functional Framework and Preliminaries}\label{2}
In this section, we introduce the function spaces and collect some fundamental estimates in the waveguide setting to make the article self-contained. This part is the waveguide version of Section 2 in \cite{MR4665720}, thus we will make it concise by omitting the proofs. 

First, for convenience, we rewrite the original Zakharov system as a first-order system as follows, which is a standard reduction in the study of the Zakharov system (this reduction allows one to use the contraction mapping method in an easier way). See \cite{MR4665720,bejenaru20092d,ginibre1995generalized,kishimoto2013local}. 

Let $w = v + i \langle \nabla \rangle^{-1} \partial_t v$ and $w_0 = v_0 + i \langle \nabla \rangle^{-1} v_1$. 
Then, we may rewrite the system \eqref{Zakharov} as
\begin{equation}\label{Zakharov2}
\begin{cases}
i \partial_t u  + \Delta u = \frac12 (w + \overline{w})u,\\
i \partial_t w - \langle \nabla \rangle w = - \langle \nabla \rangle^{-1} \Delta (|u|^2) - \langle \nabla \rangle^{-1}\bigl( \frac{w + \overline{w}}{2} \bigr),\\
(u(0), w(0)) = (u_0,w_0) \in H^s(\mathbb{R}^m\times \mathbb{T}^n) \times H^{\ell}(\mathbb{R}^m\times \mathbb{T}^n).
\end{cases}
\end{equation}
We note that the local well-posedness of \eqref{Zakharov2} in $H^s(\mathbb{R}^m\times \mathbb{T}^n) \times H^{\ell}(\mathbb{R}^m\times \mathbb{T}^n)$ implies that of \eqref{Zakharov} in $\mathcal{H}^{s,\ell}(\mathbb{R}^m\times \mathbb{T}^n)$. 
Thus, it suffices to investigate system \eqref{Zakharov2} instead of system \eqref{Zakharov}.

\begin{definition}
Let $\eta :\R \to \R$ be a smooth function satisfying $\eta=1$ on $[-1,1]$ and $\textmd{supp } \eta \subset (-2,2)$. Let $N \in 2^{\N_0}$ with $\N_0 = \N \cup \{0\}$. Define
\[
\eta_1 = \eta, \quad \eta_{N}(r) = \eta \Bigl( \frac{r}{N} \Bigr) - \eta \Bigl( \frac{2r}{N}\Bigr) \quad (N \geq 2).
\]
$\{P_N\}_{N \in 2^{\N_0}}$ denotes the collection of standard Littlewood--Paley operators defined by $P_N = \mathcal{F}_{k}^{-1} \eta_{N}(|k|) \mathcal{F}_x$. Here, \(\mathcal{F}_x\) and \(\mathcal{F}^{-1}_x\) denote the Fourier transform and its inverse on \(\mathbb{R}^m \times \mathbb{T}^n\). 

Let $\widetilde{u} (\tau,k)= \mathcal{F}_{t,x} u (\tau,k)$, $L \in 2^{\N_0}$ and
\[
Q_L^S u =  \mathcal{F}_{t,x}^{-1} \big( \eta_{L}(\tau + |k|^2) \widetilde{u} \big), \qquad Q_L^{W_{\pm}} u = \mathcal{F}_{t,x}^{-1}  \big( \eta_{L}(\tau\pm \langle k \rangle) \widetilde{u}\big).
\]
We write $P_{N,L}^S = P_{N} Q_L^S$ and $P_{N,L}^{W_{\pm}} = P_N Q_{L}^{W_{\pm}}$.

We define the function spaces $X_S^{s,b}$ and $X_{W_{\pm}}^{s,b}$ as follows:
\begin{align*}
& X_S^{s,b} = \bigl\{ u \in \mathcal{S}'(\R \times \mathbb{R}^m\times \mathbb{T}^n) \, : \,  \|u\|_{X_S^{s,b}} = \bigl( \sum_{N,L} L^{2b} N^{2s} \| P_{N,L}^S u \|_{L_{t,x}^2}^2 \bigr)^{\frac12} < \infty\bigr\},\\
& X_{W_{\pm}}^{s,b} = \bigl\{ u \in \mathcal{S}'(\R \times \mathbb{R}^m\times \mathbb{T}^n) \, : \,  \|u\|_{X_{W_{\pm}}^{s,b}} = \bigl( \sum_{N,L} L^{2b} N^{2s} \| P_{N,L}^{W_{\pm}} u \|_{L_{t,x}^2}^2 \bigr)^{\frac12} < \infty \bigr\}.
\end{align*}
Let $T>0$ and $X$ be either $X_S^{s,b}$ or $X_{W_{\pm}}^{s,b}$. We define the time restricted space $X(T)$ as follows:
\begin{align*}
X(T) & = \{ u \in C([0,T);H^s(\mathbb{R}^m\times \mathbb{T}^n)) \, : \,  \|u\|_{X(T)} < \infty \},\\
\|u\|_{X(T)} & = \inf  \{\|U\|_{X} \, : \, U \in X, \ \  u(t) = U(t) \ \forall  t \in (0,T) \}.
\end{align*}
Notice that, in the case $b > \frac12$, the Sobolev embedding in time implies $\|u\|_{L^{\infty}_t H^s} \lesssim \|u\|_{X^{s,b}_S}$. Thus, if $b >\frac12$, $X^{s,b}_S(T)$ is a Banach space. The same holds for $X^{s,b}_{W_{\pm}}(T)$. 
\end{definition}
We briefly include the well-known properties of $X^{s,b}$-type spaces, which will be used for the proof of the well-posedness theory. Since they are standard, we omit the proof. We refer to Tao \cite{tao2006nonlinear}.
\begin{lemma}[Transference principle]\label{lemma:TransferencePrinciple}
Let $U(t) \in \{e^{it\Delta}, e^{\mp it \langle \nabla \rangle}\}$. 
We use the notations
\[
X_{U}^{s,b}= \begin{cases} 
X_S^{s,b} & \mathrm{if} \ \ U(t) = e^{it\Delta},\\
X_{W_{\pm}}^{s,b} & \mathrm{if} \ \ U(t) = e^{\mp it \langle \nabla \rangle},
\end{cases}
\qquad Q_L^{U}  = \begin{cases} 
Q_L^{S}& \mathrm{if} \ \ U(t) = e^{it\Delta},\\
Q_L^{W_{\pm}} & \mathrm{if} \ \ U(t) = e^{\mp it \langle \nabla \rangle}.
\end{cases}
\]
Let $N \in 2^{\N_0}$ and $\phi \in L^2(\mathbb{R}^m\times \mathbb{T}^n)$. 
Assume that there exist $q$, $r \in [2,\infty]$, and $\alpha \in \R$ such that the linear estimate,
\begin{equation}\label{ass:lemma1.4-Strichartz}
\|U(t) P_N \phi \|_{L_{[-1,1]}^q L_x^r} \lesssim N^{\alpha} \|P_N \phi\|_{L_x^2}.
\end{equation}
holds. Then, for $L \in  2^{\N_0}$, we have
\begin{equation}\label{est:lemma1.4-goal}
\|  Q_{L}^{U} P_N u \|_{L_{t}^q L_x^r} \lesssim  L^{\frac12} N^{\alpha} \|Q_{L}^{U} P_N u\|_{L_{t,x}^2},
\end{equation}
where $L^q_t$ denotes $L^q_t(\mathbb{R})$.
\end{lemma}
We also have the following corollary.
\begin{corollary}\label{corollary:StrichartzS}
Let $N$, $L \in 2^{\N_0}$. Assume that $q, \, p \in [2, \infty]$ satisfy
\[
\frac{2(d+2)}{d} \leq q \leq \infty, \qquad \frac{1}{q} = \frac{d}{2} \Bigl(\frac12 - \frac1p \Bigr).
\]
Then, we have
\begin{equation}\label{est:StrichartzS}
\| P_{N,L}^S u\|_{L_t^q L_x^p} \lesssim_\varepsilon L^{\frac12}N^{\varepsilon}\|P_{N,L}^S u\|_{L_{t,x}^2},
\end{equation}
for all $\varepsilon>0$. 
\end{corollary}

Next we recall the Strichartz estimates for the wave equation on $ \R^{d}$. We refer to \cite{ginibre1995generalized,keel1998endpoint}.
\begin{theorem}[Strichartz estimate for wave equations]\label{theorem:WaveStrichartzR^d}
Let $d \geq 2$ and assume that $q, \, p \in [2, \infty]$ satisfy
\begin{equation}\label{ass:WaveAdmissible}
\frac1q = \frac{d-1}{2} \Bigl(\frac12 - \frac1p \Bigr), \quad  (q,p,d) \not=(2,\infty,3).
\end{equation}
Then, we have

\begin{equation}\label{est:WaveStrichartzR^d}
\|  \mathcal{W} (t)( f, g) \|_{L_{t}^q L_{x}^p(\R \times \R^d)} \\
\lesssim  \| f\|_{\dot{H}^{\frac{d}{2}- \frac{d}{p}-\frac1q}(\R^d)} +\|  g\|_{\dot{H}^{\frac{d}{2}- \frac{d}{p} - \frac1q -1}(\R^d)},
\end{equation}
where 
\[
\mathcal{W} (t) (f,g) := \cos(t |\nabla|) f + \frac{\sin(t|\nabla|)}{|\nabla|} g.
\]

\end{theorem}

\begin{remark}[Natural Extensions to the periodic/waveguide case]
It is well-known that the solution to the linear wave equation possesses the \textit{finite speed of propagation }property. We refer to Tzvetkov \cite{tzvetkov2019random} for the details. By exploiting such a property, one can derive the Strichartz estimates for the linear wave equation under the semi-periodic setting from those in the Euclidean space.\footnote{It is interesting to investigate the global-in-time Strichartz estimate for wave equations in the waveguide setting, which is somehow expected but not known yet. Global-in-time Strichartz estimates are crucial for studying the long time behavior.} Thus we see that Theorem \ref{theorem:WaveStrichartzR^d} still\textbf{ holds} locally in time if one replaces the space $\R^d$ by $\T^d$ or $\R^m\times \T^n$ ($m+n=d$).\footnote{For the waveguide case, it is even possible to show global-in-time Strichartz estimates since there are dispersions due to the Euclidean components. See \cite{Barron14} for the Schr\"odinger case.    }
\end{remark}
Similar to Section 2 in \cite{MR4665720}, combining  Lemma \ref{lemma:TransferencePrinciple} and Theorem \ref{theorem:WaveStrichartzR^d}, the following corollaries in the waveguide setting hold.
\begin{corollary}\label{corollary:StrichartzW}
Let $d=m+n \geq 2$, $L, N \in 2^{\N_0}$ and assume that $q, \, p \in [2, \infty]$ satisfy \eqref{ass:WaveAdmissible}. Then, we have
\begin{equation}\label{est:StrichartzW}
\|w_{\pm}\|_{L_t^q L_x^p} \lesssim L^{\frac12} N^{\frac{d}{2}-\frac{d}{p} - \frac{1}{q}} \|w_{\pm}\|_{L_{t,x}^2},
\end{equation}
for $w_{\pm} \in L^2(\R \times \R^m \times \T^n)$ such that
\[
\textmd{supp } \widetilde{w}_{\pm} \subset \{ (\tau,k) \in \R \times \R^m \times \Z^n \, : \, 
|\tau- \tau_0 \pm \langle k \rangle | \lesssim L, \ \  |k| \sim N \},
\]
where $\tau_0 \in \R$.
\end{corollary}

\begin{corollary}\label{corollary:StrichartzShell}
Let $d \geq 2$, $L \in 2^{\N_0}$, $N_1$, $N_2 \in 2^{\N_0}$ satisfy $N_2 \leq N_1$, and $p =\frac{2(d+1)}{d-1}$. 
Then we have
\begin{equation}\label{est:StrichartzSS}
\|u\|_{L_{t,x}^p} \lesssim_\varepsilon L^{\frac12}N_2^{\varepsilon}\|u\|_{L_{t,x}^2},
\end{equation}
for any $\varepsilon>0$ and $u \in L^2(\R \times \R^m \times \T^n)$ such that 
\[
\textmd{supp } \widetilde{u} \subset \{(\tau,k) \in \R \times \R^m \times \Z^n \, : \,  |\tau- \tau_0-|k|^2| \lesssim L, \ c_* - 1  \leq |k| \leq c_* + 1, \ k \in B_{N_2}\},
\]
where $\tau_0 \in \R$, $c_* \sim N_1$ and $B_{N_2} \subset \R^d$ is a ball of radius $N_2$ with an arbitrary center. 
\end{corollary}

\section{Strichartz Estimates on Shells in $\R^d$: Proof of Theorem \ref{mainthm on Rd}}\label{3}

In this section, we prove global-in-time shell-type Strichartz estimates on the Euclidean space $\mathbb{R}^d$, as stated in Theorem \ref{mainthm on Rd}. The strategy relies on the Stein–Tomas restriction theorem and an application of spherical coordinates.

Recall that the classical Strichartz estimate on $\mathbb{R}^d$ is a $L_x^2\to L_{t,x}^{\frac{2(d+2)}{d}}$-type estimate\footnote{See \cite{keel1998endpoint} for the classical Strichartz estimate for Schr\"odinger equations in the Euclidean setting.}, that is
$$   \|e^{it \Delta} \phi\|_{L^{\frac{2(d+2)}{d}}_{t,x}(\mathbb{R} \times \mathbb{R}^d )} \lesssim  \| \phi\|_{L_x^2(\mathbb{R}^d )}.   $$
The shell-type estimate in $d$ dimensions is concerned, which corresponds to the $(d-1)$-dimensional classical Strichartz estimate, i.e., a $L^2 \to L^{\frac{2(d+1)}{d-1}}$-type estimate. Another frequency-restricted estimate of the same category may be more familiar to the readers (\textit{strip-restricted} type); let us state it explicitly here.

\begin{proposition}\label{prop:Strichartz2}
 Assume that $a\in \mathbb{R}^d, |a|=1$, and the initial data $\phi\in L^2(\mathbb{R}^d)$ with 
$$
{\rm supp}\, \hat{\phi} 
\subset 
\{ \xi \in \mathbb{R}^d : |a\cdot \xi|\le 1\},
$$
then the estimate
\begin{equation}\label{e:Strichartz2}
 \|e^{it \Delta} \phi\|_{L^{\frac{2(d+1)}{d-1}}_{t,x}(\mathbb{R} \times \mathbb{R}^d )} \lesssim  \| \phi\|_{L_x^2(\mathbb{R}^d )}   
\end{equation}
holds.   
\end{proposition}
 We provide a brief explanation of the proof of Proposition \ref{prop:Strichartz2}. Without loss of generality, we may assume that $a=(1,0,\cdots,0)$. Observe that the space-time Fourier support of $e^{it \Delta} \phi$ is contained in $[-1,1]\times \mathcal{N}_{\mathbb{P}^{d-1}}(1)$, where 
 $$\mathcal{N}_{\mathbb{P}^{d-1}}(1)=\{(\eta, \tau)\in \mathbb{R}^{d-1}\times \mathbb{R}:\left|\tau-|\eta|^2\right|\le 1\},$$
then, by applying the Bernstein inequality and the classical 
$(d-1)$-dimensional Strichartz estimate, we could prove \eqref{e:Strichartz2}.

In broad-narrow analysis, one often needs variants of Proposition \ref{prop:Strichartz2}; we refer to \cite{MR2860188} for details. From the perspective of wave packet concentration, an analogous proposition can be found in Section $7$ of \cite{MR3702674}. 

Our main result to be proven in this section (Theorem \ref{mainthm on Rd}) is analogous to Proposition \ref{prop:Strichartz2}, differing in that we impose a distinct assumption on the Fourier support of the initial data $\phi$. Before giving the proof of Theorem \ref{mainthm on Rd}, we provide a quick remark for Proposition \ref{prop:Strichartz2}. 

We note that, it is possible to establish a stronger (more general) version of Proposition \ref{prop:Strichartz2}, i.e. establishing estimates for mixed norm $L^q_t L^p_x$ on the left hand side (full Strichartz range). These estimates allow one to establish well-posedness theory for supercritical NLS with strip-type restricted initial data, which can be regarded as applications.\footnote{For general data, it is not expected to prove well-posedness results for supercritical NLS. We refer to Christ-Colliander-Tao \cite{CCT}.} See Section \ref{new} for more details.   

\begin{proof}[\textbf{Proof of Theorem \ref{mainthm on Rd}}]
The proof relies on the Stein-Tomas restriction theorem on \(S^{d-1}\). The case \(R\lesssim 1\) is trivial; we thus focus on \(R\gg 1\). By assumptions and the change of variables, there holds
\begin{align}
e^{it \Delta} \phi(x) =& \int_{\bR^d} e(x\cdot \xi+ t|\xi|^2) \widehat{\phi}(\xi) {\rm d}\xi \nonumber\\
=&\int_{\bR} e(t r^2) \lf(   \int_{S^{d-1}} e(rx\cdot \xi^\prime) \widehat{\phi}(r\xi^\prime) {\rm d} \sigma (\xi^\prime)    \ri)  r^{d-1}\mathbf{1}_{[R-1,R+1]}(r)       {\rm d}    r .  \label{eq:mainthm on Rd-1}
\end{align}

We define 
$$f(x,r)=\lf(   \int_{S^{d-1}} e(rx\cdot \xi^\prime) \widehat{\phi}(r\xi^\prime) {\rm d} \sigma (\xi^\prime)    \ri)  r^{d-1}\mathbf{1}_{[R-1,R+1]}(r)$$
and
$$  Tg(t)=  \int_{\bR} e(t r^2) g(r) \mathbf{1}_{[R-1,R+1]}(r)       {\rm d}    r.   $$

On the one hand, we have
$$  \|Tg\|_{L^\infty_t} \lesssim \|g\|_{L^1_r},   $$
and using Plancherel’s Theorem, 
\begin{align*}
\|Tg\|_{L^2_t} = & \|\int_{\bR} e(t s) g(\sqrt{s}) \mathbf{1}_{[(R-1)^2,(R+1)^2]}(s) s^{-\frac12}       {\rm d}    s\|_{L^2_t} \\
\sim& \| g(\sqrt{s}) \mathbf{1}_{[(R-1)^2,(R+1)^2]}(s) s^{-\frac12}       \|_{L^2_s} \less R^{-1/2}\|g\|_{L^2_r},
\end{align*}
by interpolation, we get 
\begin{equation} \label{eq:mainthm on Rd-2}
    \|Tg\|_{L^p_t} \less R^{-1/p} \|g\|_{L^{p^\prime}_r}.
\end{equation}

On the other hand, by the Stein-Tomas restriction theorem on the sphere $S^{d-1}$, for fixed $r$,
\begin{align}
    \|f(x,r)\|_{L^p_x}=&\lf( \int_{\R^d}\lf|  \int_{S^{d-1}} e(rx\cdot \xi^\prime) \widehat{\phi}(r\xi^\prime) {\rm d} \sigma (\xi^\prime)     \ri|^p {\rm d}(r x) \ri)^{1/p} r^{d-1-d/p}\mathbf{1}_{[R-1,R+1]}(r) \nonumber \\
     \less & \lf( \int_{S^{d-1}} |\widehat{\phi}(r\xi^\prime)|^2 {\rm d} \sigma (\xi^\prime)  \ri)^{1/2} r^{d-1-d/p}\mathbf{1}_{[R-1,R+1]}(r), \label{eq:mainthm on Rd-3}
\end{align}
where we have used $p=\frac{2(d+1)}{d-1}$.

By \eqref{eq:mainthm on Rd-1}, \eqref{eq:mainthm on Rd-2} and Minkowski inequality, we deduce that
\begin{align*}
\|e^{it \Delta} \phi(x)\|_{L^p_x L^p_t} =& \| T(f(x,r)) (t)\|_{L^p_x L^p_t} \less  R^{-1/p} \|f(x,r)\|_{L^p_x L^{p^\prime}_r} \\
 \le &R^{-1/p} \|f(x,r)\|_{ L^{p^\prime}_r L^p_x},
\end{align*}
then we use \eqref{eq:mainthm on Rd-3}, H\"older inequality and the change of variables, there holds
\begin{align*}
& R^{-1/p} \|f(x,r)\|_{ L^{p^\prime}_r L^p_x} 
\less   R^{-1/p} \lf\|\lf( \int_{S^{d-1}} |\widehat{\phi}(r\xi^\prime)|^2 {\rm d} \sigma (\xi^\prime)  \ri)^{1/2} r^{d-1-\frac{d}{p}}\mathbf{1}_{[R-1,R+1]}(r) \ri\|_{L^{p^\prime}_r} \\
\less & R^{(d-1)/2-d/p-1/p} \lf\|\lf( \int_{S^{d-1}} |\widehat{\phi}(r\xi^\prime)|^2 {\rm d} \sigma (\xi^\prime)  \ri)^{1/2} r^{(d-1)/2}\mathbf{1}_{[R-1,R+1]}(r) \ri\|_{L^{2}_r} \\
\less & R^{(d-1)/2-d/p-1/p} \lf(  \int_0^\infty \int_{S^{d-1}} |\widehat{\phi}(r\xi^\prime)|^2 r^{d-1}  {\rm d} \sigma (\xi^\prime)  {\rm d} r        \ri)^{1/2}\sim \|\phi\|_{L_x^2}.
\end{align*}
This completes the proof of Theorem \ref{mainthm on Rd}.
\end{proof}

\begin{remark}
 In general, how different restrictions on the Fourier support of the initial data $\phi$ influence the corresponding Strichartz estimates is an interesting problem. Two natural types of restrictions (\textit{shell-type} and \textit{strip-type}) are discussed in the current paper. It is also interesting to investigate a stronger (more general) version of Theorem \ref{mainthm on Rd}, i.e. establishing estimates for mixed norm $L^q_t L^p_x$ on the left hand side in Theorem \ref{mainthm on Rd}, which is highly nontrivial.   
\end{remark}

\section{Waveguide Geometry in 3D: Estimates and Sharp Counter-examples}\label{4}
Building on the Euclidean case, we now turn to the waveguide setting, where the interplay between Euclidean and periodic components introduces new challenges. The Euclidean result  relies on the full dispersion, whereas the waveguide case requires a delicate balance between continuous and discrete frequencies. 

In this section, we present the proofs for Theorem \ref{mainthm on R2T} and Theorem \ref{counterexample on RTd-1} respectively. In principle, we intend to prove an analogous result of Theorem 1.7 in \cite{MR4665720} for the case $\mathbb{R}^2 \times \mathbb{T}$. Compared to the analogous inequalities on $\mathbb{T}^3$ , the estimate we establish here on $\mathbb{R}^2 \times \mathbb{T}$ is global-in-time and requires no derivative loss. Moreover, we discuss the derivative loss issue of the shell-type Strichartz estimate for the $\mathbb{R}\times\mathbb{T}^d$ case by giving the proof of Theorem \ref{counterexample on RTd-1}.\footnote{We note that $\mathbb{R}\times\mathbb{T}^2$ case is included when $d=2$. These results, together with Theorem 1.7 in \cite{MR4665720}, offer a complete view for shell-type Strichartz estimates on three dimensional spaces $\mathbb{R}^m\times \mathbb{T}^n$ ($m+n=3,m,n\geq 1$).} To be specific, we will demonstrate that the derivative loss is inevitable.

\subsection{The proof of Theorem \ref{mainthm on R2T}}
In contrast to the three-dimensional Euclidean case (Theorem \ref{mainthm on Rd}), our Theorem \ref{mainthm on R2T} replaces the underlying space of the initial data from $\mathbb{R}^3$ with $\mathbb{R}^2 \times \mathbb{T}$. By a standard argument as in \cite{DPST}, we will reduce the $L^2 \to L^4$-type estimate \eqref{e:ShellStrichartz} to measure estimates. Leveraging a simple yet insightful idea drawn from a different analytical setting, we make use of the machinery developed in \cite{basu2021stationarysetmethodestimating} to derive the required measure estimates. In particular, the proof relies on controlling the measure of \textit{semi-algebraic sets} in $\mathbb{R}^2\times\mathbb{Z}$. The key observation is that the Euclidean component dominates the periodic resonance, preventing derivative loss.

 Following the exposition in \cite{basu2021stationarysetmethodestimating}, we provide the definition of semi-algebraic sets and related notions. We will say that a set $U \subset \bR^{d+1}$ is a \emph{semi-algebraic set}, if $U$ is a finite union of subsets, each of which is defined by a formula of the form $P=0, Q_1>0,\cdots, Q_k>0$, where each $P, Q_1,\cdots, Q_k$ is polynomial. If the total number of polynomials that define $U$ is bounded by $s$, and the maximum degree of these polynomials is bounded by $D$, we will say that the \emph{complexity} of $U$ is bounded by $sD$. We now state a lemma concerning the upper bound of the measure of semi-algebraic sets in $\mathbb{R}^d \times \mathbb{Z}$.

\begin{remark}[Intuition]
The key is controlling $\sum_{n\in\mathbb{Z}}\int_{\mathbb{R}^2}\mathbf{1}_U(\xi,n)d\xi$ via \textbf{measure-theoretic transversality}: the semi-algebraic set $U$ intersects lattice fibers $\mathbb{R}^2\times\{n\}$ with controlled cardinality.
\end{remark}

\begin{lemma}\label{lem:control semialgebraic set}
Let $d \ge 1$ and let $U$ be a bounded semi-algebraic set on $\bR^{d+1}$ with complexity $\less 1$. Then
$$  \sum_{n\in \Z}\int_{\bR^d } \mathbf{1}_{U}(\xi,n) {\rm d}\xi \less |U|+       \sup_{n\in \Z}\int_{\bR^d } \mathbf{1}_{U}(\xi,n) {\rm d}\xi,      $$
where $|U|$ denotes the Lebesgue measure of $U$ on $\bR^{d+1}$.    
\end{lemma}

The meaning of the aforementioned lemma can be intuitively illustrated through simple schematic diagrams. See the figure as below. Let the red region in Figure $1$ represent an arbitrary semi-algebraic set $U$. 
Lemma \ref{lem:control semialgebraic set} asserts that the total length of intersections between the blue dashed lines and $U$ 
can be controlled by the total measure of $U$ and maximal slice measure.

\begin{center}
\begin{tikzpicture}[scale=1.2]

    \draw[thick, red, fill=red!20, fill opacity=0.3] 
        plot[smooth cycle, tension=0.8] coordinates {
            (-3,1) (-1,2) (1,1.5) (3,0.5) 
            (2,-1) (0,-2) (-2,-1) (-3,1)
        };

    \foreach \y in {-2.4,-1.6,-0.8,0.0,0.8,1.6,2.4} {
        \draw[blue, dashed] (-4.5,\y) -- (4.5,\y);
    }
    \node at (0,-3){Figure 1};
   \end{tikzpicture}
   \end{center}
(Red region: a semi-algebraic set $U$; blue lines: lattice points $\mathbb{Z}$.)

We shall first assume the validity of Lemma \ref{lem:control semialgebraic set} to establish Theorem \ref{mainthm on R2T}. We note that the set \( U \) is defined by polynomials of degree \( \leq 2 \) (from quadratic forms \( (\xi - a) \cdot (\xi - b) \)), hence its complexity is uniformly bounded.

\begin{proof}[\textbf{Proof of Theorem \ref{mainthm on R2T}}]
It suffices to prove that
$$\|e^{it \Delta} \phi\|_{L^4_{t,x}([0,T] \times \mathbb{R}^2 \times \mathbb{T})} \lesssim  \| \phi\|_{L_x^2(\mathbb{R}^2 \times \mathbb{T})}$$
holds for all $T>1$, $\phi \in L^2(\mathbb{R}^2 \times \mathbb{T})$ satisfying 
$$
{\rm supp}\, \hat{\phi} 
\subset 
\{ (\xi, n) \in \mathbb{R}^2\times \Z : c_*-1 \le |(\xi, n)| \le c_* + 1\} \cap B_{c_*/100},
$$
where $B_{c_*/100} \subset \R^3$, a ball of radius $c_*/100$ with an arbitrary center.

 Define $V=\{ (\xi, n) \in \mathbb{R}^2\times \Z : c_*-1 \le |(\xi, n)| \le c_* + 1\} \cap B_{c_*/100}$, and we assume that $V$ is nonempty.   Then, by a standard argument in the proof of Proposition 3.7 in \cite{DPST}, it suffices to prove that, for any fixed $a, b\in V $, there holds 
 $$   \sum_{n \in \Z} \int_{\bR^2} \mathbf{1}_U(\xi, n) {\rm d}\xi\less \frac1T,  $$
 where 
 $$     U=\{ \w{\xi}\in \bR^3: c_*-1 \le |\w{\xi}| \le c_* + 1, |(\w{\xi}-a)\cdot(\w{\xi}-b)|\le \frac1T  \}.$$

A direct calculation shows 
  $$\sup_{n\in \Z}\int_{\bR^2 } \mathbf{1}_{U}(\xi,n) {\rm d}\xi \less \frac1T.  $$
  After applying Lemma \ref{lem:control semialgebraic set} for the set $U$, it suffices to prove
  $$  |U| \less \frac1T.$$

 Making a rotation, we could assume that
  $$  \frac{a+b}{2}=(\alpha, 0, 0), \quad \alpha>0.  $$
  Observe that $a,b \in V$, so
  $$  c_*-1 \le |a|, |b| \le c_* + 1, \quad |\frac{a-b}{2}|\le   c_*/100 .     $$
  By some computations, we can deduce that
  $$  \alpha \gtrsim c_* .$$
Let $\w{\xi}=(\alpha+r\cos \theta, r\sin \theta \cos \varphi, r\sin \theta \sin \varphi)$, where $r\ge 0, 0\le \theta \le \pi$ and $ 0\le \varphi \le 2\pi$. Define $\beta=|\frac{a-b}{2}|$. Then 
$$ \w{\xi}\in U \Longleftrightarrow |r^2-\beta^2|\le \frac1T, (c_*-1)^2\le \alpha^2+2\alpha r \cos \theta+r^2\le (c_*+1)^2 . $$

 If $\beta^2 > \frac1T$, we have 
\begin{align*}
 |U|  &= 2\pi \int_{\sqrt{\beta^2 - 1/T}}^{\sqrt{\beta^2 + 1/T}} \int_0^\pi \mathbf{1}_{\{ \alpha^2+2\alpha r \cos \theta+r^2\in [(c_*-1)^2, (c_*+1)^2] \}}(r,\theta) r^2 \sin \theta {\rm d} \theta {\rm d} r \\
 &= 2\pi \int_{\sqrt{\beta^2 - 1/T}}^{\sqrt{\beta^2 + 1/T}} \int_{-r}^r \mathbf{1}_{\{ \alpha^2+2\alpha s +r^2\in [(c_*-1)^2, (c_*+1)^2] \}}(r,s) r {\rm d} s {\rm d} r \\
 &\less \int_{\sqrt{\beta^2 - 1/T}}^{\sqrt{\beta^2 + 1/T}} \frac{c_*}{\alpha} r {\rm d} r.
\end{align*}
Note that $\alpha \gtrsim c_*$, then we get $|U| \less \frac1T$. if $\beta^2 \le\frac1T$, by a similar computation, we can also prove $|U|\less \frac{1}{T}$. 

\end{proof}

Now, it remains to prove Lemma \ref{lem:control semialgebraic set}. We combine the following two lemmas to achieve this goal. 

\begin{lemma} \label{lem:control semialgebraic set-1}
Suppose $I\subset \bR$ is a interval. Let $F\in C(I)$ be a non-negative function. If $F$  changes monotonicity $O(1)$ times, then
$$   \sum_{n \in I\cap \Z} F(n) \less \int_I F(\eta) {\rm d}\eta +  \sup_{n \in I\cap \Z} F(n).            $$
\end{lemma}

\begin{lemma}[Lemma 2.9 in \cite{basu2021stationarysetmethodestimating}] \label{lem:control semialgebraic set-2}
Let $d \ge 1$ and let $U$ be a bounded semi-algebraic set on $\bR^{d+1}$ with complexity $\le k$. Then the function
$$   F(\eta):=    \int_{\bR^d } \mathbf{1}_{U}(\xi,\eta) {\rm d}\xi       $$
changes monotonicity $O_{d,k}(1)$ times.

\end{lemma}

To establish Lemma \ref{lem:control semialgebraic set-1}, it suffices to observe that when the non-negative function $F$ is non-decreasing, 
the inequality 
  $F(n)\le  \int_{n}^{n+1} F(x) {\rm d}x$  holds. The proof of Lemma \ref{lem:control semialgebraic set} thereby reduces to the direct application of Lemma \ref{lem:control semialgebraic set-1} and Lemma \ref{lem:control semialgebraic set-2}.

\begin{remark}
 The application of Lemma \ref{lem:control semialgebraic set} in the proof of Theorem \ref{mainthm on R2T} allows us to systematize measure-theoretic computations for Strichartz estimates on waveguides, particularly encompassing the key calculations underlying the main theorems in \cite{MR4782142, takaoka20012d, MR4219972}.

 The counter-example in Theorem \ref{counterexample on RTd-1} exploits the lack of dispersion in the \(\mathbb{T}^{2}\) component, which allows energy concentration at high frequencies. In contrast, two Euclidean dimensions (\(\mathbb{R}^{2}\)) provide sufficient dispersion to prevent this.
\end{remark}

\subsection{A counter-example for the $\mathbb{R}\times \mathbb{T}^{d-1}$ case}
We now present the proof for Theorem \ref{counterexample on RTd-1}. We will borrow the idea from the example constructed in \cite{MR4665720} for the periodic case with suitable modifications, to illustrate that the Strichartz estimate on the shell for $\bR \times \T^{d-1}$ does not hold.\footnote{We note that, in principle, shell-type Strichartz estimate on $\bR \times \T^{d-1}$ can be compared with Strichartz estimate on $\T^{d-1}$, where there must be a derivative loss (see \cite{Bourgain13,takaoka20012d} for counter-examples). We also note that it is interesting to prove ``logarithm-type" derivative loss shell-type Strichartz estimates on $\bR \times \T^{d-1}$ or  on $\T^{d}$. We refer to a recent breakthrough Herr-Kwak \cite{Herr} for more details.} (The case $\bR \times \T^{2}$ is covered.)

\begin{proof}[\textbf{Proof of Theorem \ref{counterexample on RTd-1}}]
We assume that $c_*=N_1$ without loss of generality.   Let 
\begin{equation*}
\hat{\phi}(\xi, n)=\begin{cases}
    1, & \mbox{if~} \xi\in (N_1-\frac{1}{100}, N_1+\frac{1}{100}), 1\le n_j \le d^{-\frac{1}{2}}N_2^{\frac12}, n_j\in \Z, j=1,\cdots, d-1, \\
    0, & \mbox{otherwise}.
\end{cases}
\end{equation*}
Then $\|\phi\|_{L^2}\sim N_2^{\frac{d-1}{4}}$. We have
\begin{align*}
&\|e^{it \Delta} \phi\|_{L^p_{t,x}([0,1] \times \mathbb{R} \times \mathbb{T}^{d-1})}^p \\
= & \int_0^1 \lf(  \lf(\int_\bR \lf|\int_{N_1-\frac{1}{100}}^{N_1+\frac{1}{100}} e^{2\pi i(x_1 \xi+t \xi^2)} {\rm d}\xi \ri|^p {\rm d}x_1 \ri) \lf(\int_{\mathbb T} \lf|\sum_{1\le k \le d^{-\frac12}N_2^{\frac12}} e^{2\pi i(y k+t k^2)}  \ri|^p {\rm d}y \ri)^{d-1}  \ri) {\rm d} t.
\end{align*}
Note that, when $t\in [0,1]$ is fixed, the integral
\begin{align*}
  \int_\bR \lf|\int_{N_1-\frac{1}{100}}^{N_1+\frac{1}{100}} e^{2\pi i(x_1 \xi+t \xi^2)} {\rm d}\xi \ri|^p {\rm d}x_1 = \int_\bR \lf|\int_{-\frac{1}{100}}^{\frac{1}{100}} e^{2\pi i(s \eta+t \eta^2)} {\rm d}\eta \ri|^p {\rm d}s \gt 1,
\end{align*}
here we make a change of variables $s=x_1+2N_1t, \eta=\xi-N_1$. Thus
\begin{align*}
    &\|e^{it \Delta} \phi\|_{L^p_{t,x}([0,1] \times \mathbb{R} \times \mathbb{T}^{d-1})}^p\\
\gt & \int_0^1    \lf(\int_{\mathbb T} \lf|\sum_{1\le k \le d^{-\frac12}N_2^{\frac12}} e^{2\pi i(y k+t k^2)}  \ri|^p {\rm d}y \ri)^{d-1}   {\rm d} t \\
\gt & (\log N_2) N_2^{\frac{d+1}{2}} \sim (\log N_2) \|\phi\|_{L^2}^p.
\end{align*}
See Theorem 13.6 in \cite{MR3971577} for the last inequality.

It shows that there is at least an inevitable log-derivative loss. This contradicts the assumed bound, thus proving the necessity of derivative loss.
\end{proof}

\subsection{Conclusion for the 3D case}
We now have four specific cases: $\mathbb{R}^3$, $\mathbb{R}^2 \times \mathbb{T}$, $\mathbb{R} \times \mathbb{T}^2$, and $\mathbb{T}^3$.

The shell-type Strichartz estimate for $\mathbb{R}^3$ is global-in-time and there is no derivative loss; the shell-type Strichartz estimate for $\mathbb{R}^2 \times \mathbb{T}$ is global-in-time, and there is no derivative loss; the shell-type Strichartz estimate for $\mathbb{R} \times \mathbb{T}^2$ is local-in-time and there must be a derivative loss (see the next section for the proof of the estimate with a derivative loss); the shell-type Strichartz estimate for $\mathbb{T}^3$ is local-in-time, and there must be a derivative loss.

We summarize the distinctions between different geometries in Table \ref{tab:3d}, which provides a clear analytical landscape. This completes the characterization of the shell-type Strichartz estimates in three-dimensional product geometries, and provides the analytical backbone for the well-posedness theory.
\begin{table}[h]
\centering
\caption{Comparison of shell-type Strichartz estimates in three-dimensional geometries}
\label{tab:3d}
\begin{tabular}{@{}lllll@{}}
\toprule
\textbf{Setting} & \textbf{Global/Local} & \textbf{Derivative Loss} & \textbf{Key Technique} \\ \midrule
$\mathbb{R}^3$ & Global & No & Stein-Tomas \\
$\mathbb{R}^2 \times \mathbb{T}$ & Global & No & Semi-algebraic sets \& measure estimate \\
$\mathbb{R} \times \mathbb{T}^2$ & Local & Yes ($\epsilon$) \& inevitable & Decoupling \& counter-example  \\ 
$\mathbb{T}^3$ & Local & Yes ($\epsilon$) \& inevitable & Decoupling \& counter-example\\
\bottomrule
\end{tabular}
\end{table}

\begin{remark}
The dichotomy between \(\mathbb{R}^{2}\times\mathbb{T}\) (no derivative loss) and \(\mathbb{R}\times\mathbb{T}^{2}\) (\(\epsilon\)-loss) suggests a deeper geometric principle: dispersion dominates confinement when the Euclidean dimension \(\geq 2\). This aligns with physical intuition in optical fiber design, where higher-dimensional dispersion suppresses signal distortion.    
\end{remark}

\section{Other Dimensional Shell-Type Estimates: Proof of Theorem \ref{mainthm on RmTn}}\label{5}

Having established sharp shell-type Strichartz estimates and counter-examples in the three-dimensional waveguide setting, we now generalize our framework to higher dimensions, analyzing the behavior on $\mathbb{R}^m \times \mathbb{T}^n$ for arbitrary $m,n \geq 1$.

In this section, we establish $\varepsilon$-derivative-loss shell-type Strichartz estimates for general waveguide manifolds (i.e., proving Theorem \ref{mainthm on RmTn}) based on the decoupling inequality from \cite{MR4665720}. This result recovers the periodic shell-type Strichartz estimates for the general waveguide case.
\subsection{The proof of Theorem \ref{mainthm on RmTn}}

We recall the following result in \cite{MR4665720}.
\begin{proposition}[Corollary $3.3$ in \cite{MR4665720}]\label{Cor:TwoScaleDecoup}
                Let $d\ge 2$, $N_1 \ge N_2\ge1$, $p=\frac{2(d+1)}{d-1}$ and take arbitrary $d_* \in[1,2]$. 
                Then for arbitrarily small $\varepsilon$, 
                \begin{equation}\label{e:N_2Gain}
                \big\| E_{\mathbb{P}^d} f \big\|_{L^p(B_{N_1^2})} \le C_\varepsilon N_2^{\varepsilon} \big( \sum_{\theta \in \mathcal{C}_{N_1^{-1}}} \big\| E_{\mathbb{P}^d} f_\theta  \big\|_{L^p(w_{B_{N_1^2}})}^2  \big)^\frac12
                \end{equation}
                holds for all $f$ satisfying 
                \begin{equation}\label{e:Assumpg2}
                {\rm supp}\, f\subset \{\xi\in\mathbb{R}^d: d_*-\frac1{N_1}\le |\xi|\le d_* + \frac1{N_1}\} \cap B_{N_2/N_1},
                \end{equation}
                where $B_{r} \subset \mathbb{R}^d$ is defined as a ball of radius $r>0 $ with arbitrary center,  $\mathcal{C}_{N^{-1}}$ is defined as a family of disjoint $\frac1N \times \cdots \times \frac1N $ cubes of the form \begin{equation}\label{e:DefCap}
\theta = 
\Bigl\{ 
\xi\in \bR^d : \xi \in [- \frac1{2N}, \frac1{2N}]^d + c_\theta 
\Bigr\},
\end{equation}
 $c_\theta$ runs over $\frac1N\mathbb{Z}^d \cap [-1,1]^d$, $f_\theta$ is defined by $f_\theta=f\cdot \mathbf{1}_{\theta}$ and $w_{B_{N_1^2}}$ is a weight adapted to the ball $B_{N_1^2}$. 
\end{proposition}

The derivation of the corresponding Strichartz estimates on waveguide manifolds (or torus) from the decoupling inequality is now standard; see \cite{BD15,Bourgain13,Barron14} for more details. We will briefly outline the procedure as follows.

\begin{proof}[\textbf{Proof of Theorem \ref{mainthm on RmTn}}]
It suffices to focus on the case $c_*=N_1$ and $p=\frac{2(d+1)}{d-1}$. We first assume that the inequality
\begin{equation}\label{e:decoupling lem}
    \|   e^{it \Delta} \phi   \|_{L^p_{t,x}([-1,1] \times \mathbb{R}^m \times \mathbb{T}^n)} \lesssim_\varepsilon N_2^\varepsilon  \left(\sum_{\vartheta\in \Theta} \| e^{it \Delta} \phi_\vartheta\|_{L^p_{t,x}([-1,1] \times \mathbb{R}^m \times \mathbb{T}^n)}^2\right)^{\frac{1}{2}} 
\end{equation}
holds, where 
$\Theta=\{\vartheta=[-\frac12,\frac12]^m\times \{0\}^n+k: k\in \Z^d \}$
and
$\phi_\vartheta=(\widehat{\phi} \cdot \mathbf{1}_\vartheta)^\vee$. For fixed $\vartheta \in \Theta$,  by interpolation of $L^2$ mass conservation and the trivial $L^1 \to L^\infty$ estimate and using H\"older inequality, we have
$$        \| e^{it \Delta} \phi_\vartheta\|_{L^p_{t,x}([-1,1] \times \mathbb{R}^m \times \mathbb{T}^n)} \lesssim \|\widehat{\phi_\vartheta}\|_{L^{p^\prime}} \lesssim \|\phi_\vartheta\|_{L^2} ,          $$
where $p^\prime$ satisfies $\frac{1}{p}+\frac{1}{p^\prime}=1$.  Combining the above inequality and \eqref{e:decoupling lem}, we complete the proof of Theorem \ref{mainthm on RmTn}. 

Now it suffices to prove \eqref{e:decoupling lem}. By Minkowski inequality(or parallel decoupling lemma), we only need to prove 
\begin{equation}\label{e:decoupling lem2}
\|   e^{it \Delta} \phi   \|_{L^p_{t,x}([-1,1] \times [-N_1,N_1]^m \times \mathbb{T}^n)} \lesssim_\varepsilon N_2^\varepsilon  \left(\sum_{\vartheta\in \Theta} \| e^{it \Delta} \phi_\vartheta\|_{L^p_{t,x}(w)}^2\right)^{\frac{1}{2}}, 
\end{equation}
where $w$ is a smooth weight adapted to $[-1,1] \times [-N_1,N_1]^m$. Let
$$   f(\eta_1, \eta_2)=\sum_{n\in \Z^n} \widehat{\phi}(N_1 \eta_1,n) \delta_{n}(N_1 \eta_2),$$
where $(\eta_1,\eta_2)\in \bR^m\times \bR^n, \delta_n$ denotes the Dirac measure at $n$ on $\bR^n$. Then, using Proposition \ref{Cor:TwoScaleDecoup} to the function $f$ and periodicity, we could check \eqref{e:decoupling lem2} by some computations.
\end{proof}

\begin{remark}
To ensure rigor, our approach requires approximating the Dirac measure through smooth functions, a procedure detailed in \cite{BD15,Bourgain13,Barron14}.
\end{remark}

\subsection{Conclusion for the 2D case}
We now have three cases: $\mathbb{R}^2$, $\mathbb{R} \times \mathbb{T}$ and $\mathbb{T}^2$.

The shell-type Strichartz estimate for $\mathbb{R}^2$ is global-in-time and there is no derivative loss; the shell-type Strichartz estimate for $\mathbb{R} \times \mathbb{T}$ is local-in-time and there must be a derivative loss; the shell-type Strichartz estimate for $\mathbb{T}^2$ is local-in-time and there must be a derivative loss.

We refer to the following table for a brief summary, which completes the characterization of the shell-type Strichartz estimates in two-dimensional product geometries, and provides the analytical backbone for the well-posedness theory.
\begin{table}[h]
\centering
\caption{Comparison of shell-type Strichartz estimates in two-dimensional geometries}
\begin{tabular}{@{}lllll@{}}
\toprule
\textbf{Setting} & \textbf{Global/Local} & \textbf{Derivative Loss} & \textbf{Key Technique} \\ \midrule
$\mathbb{R}^2$ & Global & No & Stein-Tomas \\
$\mathbb{R} \times \mathbb{T}$ & Local & Yes ($\epsilon$) \& inevitable  & Decoupling \& counter-example  \\ 
$\mathbb{T}^2$ & Local & Yes ($\epsilon$) \& inevitable & Decoupling \& counter-example \\
\bottomrule
\end{tabular}
\end{table}
\subsection{Conclusion for the higher dimensional case ($d\geq4$)}
We consider the following four exhaustive cases: $\mathbb{R}^d$, $\mathbb{R}^m \times \mathbb{T}^{d-m}$ ($m\geq 2$), $\mathbb{R} \times \mathbb{T}^{d-1}$, $\mathbb{T}^d$.

The shell-type Strichartz estimates for $\mathbb{R}^m \times \mathbb{T}^n$ ($m\geq 2$) can be obtained, compared to the periodic case (which means the estimates are local-in-time and there is a derivative loss). In particular, the shell-type Strichartz estimates for $\mathbb{R} \times \mathbb{T}^{d-1}$ and $\mathbb{T}^{d}$ must have a derivative loss according to Theorem \ref{counterexample on RTd-1} and \cite{MR4665720}. 

We refer to the following table for a brief summary. We will make a summary of the shell-type Strichartz estimates in different settings and provide more remarks in Section \ref{7}.
\begin{table}[h]
\centering
\caption{Comparison of shell-type Strichartz estimates in higher-dimensional geometries}
\begin{tabular}{@{}lllll@{}}
\toprule
\textbf{Setting} & \textbf{Global/Local} & \textbf{Derivative Loss} & \textbf{Key Technique} \\ \midrule
$\mathbb{R}^d$ & Global & No & Stein-Tomas \\
$\mathbb{R}^m \times \mathbb{T}^{d-m}$ ($m\geq 2$) & Local & Yes ($\epsilon$) \& (inevitable?) & Decoupling method  \\ 
$\mathbb{R} \times \mathbb{T}^{d-1}$ & Local & Yes ($\epsilon$) \& inevitable & Decoupling \& counter-example  \\ 
$\mathbb{T}^d$ & Local & Yes ($\epsilon$) \& inevitable & Decoupling \& counter-example \\
\bottomrule
\end{tabular}
\end{table}

\section{Local Well-Posedness of Zakharov Systems on Waveguides: Proof of Theorem \ref{mainthm2}}\label{6}
In this section, based on the established shell-type Strichartz estimates, we present the proof of Theorem \ref{mainthm2} by demonstrating that the first order system \eqref{Zakharov2} is locally well-posed in $H^s (\T^d) \times H^{s-\frac12}(\T^d)$ if $s > s_0$.\footnote{It is interesting to investigate if one can establish the well-posedness theory when $s=s_0$. We leave it for future studies.}This follows the periodic framework but with necessary modifications for the semi-periodic (waveguide) setting.

In a standard way, we write
\[
\mathcal{J}_S[F](t) = -i \int_0^t e^{i(t-t')\Delta}F(t') d t', \quad 
\mathcal{J}_{W_{\pm}}[G](t) = i \int_0^t e^{\mp i(t-t')\langle \nabla \rangle}  G(t') d t,
\]
and rewrite the system \eqref{Zakharov2} in integral form: 
\begin{align}
u(t) &= e^{it\Delta} u_0  + \frac{1}{2} \mathcal{J}_S[(w + \overline{w})u](t),\\
w(t) & = e^{-it \langle \nabla \rangle} w_0 + \mathcal{J}_{W_{+}} \Bigl[\frac{\Delta}{\langle \nabla \rangle} (u \overline{u}) + \frac{1}{2 \langle \nabla \rangle} (w + \overline{w}) \Bigr](t).
\end{align}

We note that the following bilinear estimates play a crucial role in the proof of Theorem \ref{mainthm2}. The proof is based on the shell-type Strichartz estimate in the waveguide setting (Theorem \ref{mainthm on RmTn}) established in the preceding sections, the Strichartz estimate for wave equations in the waveguide setting, the idea of the frequency decomposition method, and other fundamental estimates.
\begin{proposition}
\label{proposition:KeyBilinearEst}
Let $s_0$ be as defined in~\eqref{assumption:regularity} and $s>s_0$. Then there exist $b>\frac12$ and $\delta >0$ such that 
\begin{align}
& \|u w\|_{X_S^{s,b-1+ \delta}} + \|u \overline{w}\|_{X_S^{s,b-1+\delta}} \lesssim \|u \|_{X_S^{s,b}} \|w\|_{X_{W_{+}}^{s-\frac12,b}}, \label{est:prop1.3-1}\\
& \|u \overline{u}\|_{X_{W_{+}}^{s+\frac12,b-1 + \delta}} \lesssim \|u\|_{X_S^{s,b}}^2.\label{est:prop1.3-2}
\end{align}
\end{proposition}

\begin{proof}
By duality and dyadic decomposition of the space-time Fourier supports of the functions, 
to prove \eqref{est:prop1.3-1} and \eqref{est:prop1.3-2}, it suffices to prove that if $s > s_0$, there exists small $\delta >0$ such that
\begin{equation}\label{est:prop1.3-3}
\begin{split}
&\Bigl| \int u_{1} \overline{v}_{2} w_{3,\pm} dt dx \Bigr|\\
& \lesssim (L_1L_2 L_3)^{\frac12-\delta} \Bigl(\frac{\min(N_1,N_2)}{\max(N_1,N_2)} \Bigr)^{\frac12}N_{\min}^{s-\frac12} \|u_{1}\|_{L_{t,x}^2} \|v_{2}\|_{L_{t,x}^2} \|w_{3,\pm} \|_{L_{t,x}^2},
\end{split}
\end{equation}
where $u_1=  P_{N_1, L_1}^S u$, $v_2 = P_{N_2,L_2}^S v$, $w_{3,\pm} = P_{N_3,L_3}^{W_{\pm}} w$, and $N_{\min}= \min (N_1,N_2,N_3)$. 

Then it suffices to do a case-by-case analysis respectively:

(1) \underline{$N_3 \lesssim N_1\sim N_2$}; \quad (2) \underline{$\min(N_1,N_2) \ll  N_3$}.

The proof follows the same line as the periodic case in Section 4 of \cite{MR4665720} with minor adaptations to the waveguide setting once we have the shell-type Strichartz estimates and the preliminaries in Section \ref{2}, so we omit the proof.
\end{proof}
Now we turn to the proof of Theorem \ref{mainthm2}. Since the proof is quite standard and almost identical to the periodic case (see Section 4 of \cite{MR4665720}), we only present
 a rough sketch of it. The bilinear Strichartz estimate Proposition \ref{proposition:KeyBilinearEst} and the Strichartz estimate for nonlinear wave equation (and NLS\footnote{It is due to Barron \cite{Barron14}.}) on waveguide manifolds are used. 

We first introduce the well-known linear estimates. 
\begin{lemma}
\label{lemma:X^{s,b}Linear}
Let $s, b \in \R$ and $0<T <1$. Then,
\begin{align*}
\|e^{it \Delta} u_0 \|_{X^{s,b}_S(T)} \lesssim \|u_0\|_{H^s(\R^m \times \T^n)},\\
\|e^{-it \langle \nabla \rangle} w_0 \|_{X^{s,b}_{W_+}(T)} \lesssim \|w_0\|_{H^s(\R^m \times \T^n)}.
\end{align*}
\end{lemma}
\begin{remark}
We refer to Barron \cite{Barron14} for Strichartz estimates in waveguide setting. Locally (in time), it is the same as the periodic case. See also the references therein for previous related works.    
\end{remark}
Next, we state the standard estimates for handling the Duhamel terms.  
\begin{lemma}
\label{lemma:X^{s,b}Duhamel}
Let $s \in \R$, $b>\frac12$, $0<T <1$, $\delta >0$ and $\psi \in C_0^{\infty}(\R)$ satisfy $\psi = 1$ on $[-1,1]$ and $\textmd{supp } \psi \subset (-2 , 2)$. 
Define $\psi_T(t) = \psi(\frac{t}{T})$. Then,
\begin{align*}
\| \psi_T \mathcal{J}_S[F] \|_{X^{s,b}_S} & \lesssim T^{\delta} \| F\|_{X_S^{s,b-1+\delta}},\\
\| \psi_T \mathcal{J}_{W_{+}}[G] \|_{X^{s,b}_{W_{+}}} & \lesssim T^{\delta} \|  G\|_{X_{W_{+}}^{s,b-1+\delta}}.
\end{align*}
\end{lemma}
By combining Proposition \ref{proposition:KeyBilinearEst} and Lemma \ref{lemma:X^{s,b}Duhamel}, we obtain the following:
\begin{lemma}
\label{lemma:X^{s,b}Nonlinear}
Let $s > s_0$ and $0<T<1$. Then there exists $b >\frac12$ and $\delta>0$ such that
\begin{align}
\label{est:X^{s,b}NonlinearS}
& \|\mathcal{J}_S[(w + \overline{w})u] \|_{X^{s,b}_S(T)} 
\lesssim T^{\delta} \|u \|_{X_S^{s,b}(T)} \|w\|_{X_{W_{+}}^{s-\frac12,b}(T)},\\
\label{est:X^{s,b}NonlinearW}
& \Bigl\|\mathcal{J}_{W_{+}} \Bigl[\frac{\Delta}{\langle \nabla \rangle} (u \overline{u})  \Bigr] \Bigr\|_{X^{s,b}_{W_{+}}(T)} \lesssim T^{\delta} \|u\|_{X_S^{s,b}(T)}^2.
\end{align}
\end{lemma}

Now, by applying Lemmas \ref{lemma:X^{s,b}Linear} and \ref{lemma:X^{s,b}Nonlinear} with suitable exponents $T$, $\delta$, we can verify the standard contraction mapping in some ball of suitable 
radius in $X^{s,b}_S(T) \times X^{s-\frac12,b}_{W_+}(T)$ centered at the origin as the periodic case. We omit the details.

\section{Well-Posedness for Supercritical NLS with Restricted Frequency Support}\label{new}
This section investigates how frequency restrictions affect well-posedness of NLS in supercritical regimes. We  prove a stronger version of Proposition \ref{prop:Strichartz2} with full Strichartz range, which allows one to obtain some direct PDE applications. This part has its own interests. 
\subsection{On the strip-type Strichartz estimates}
We first demonstrate the following lemma, which is an analogue of the standard linear dispersive estimate in the restricted setting. We consider strip-restriction ($M>0$ instead of $1$): $$
{\rm supp}\, \widehat{\phi} 
\subset 
\{ \xi \in \mathbb{R}^d : |a\cdot \xi|\le M\}.
$$
\begin{lemma}[Dispersive estimate in the frequency-restricted setting]\label{LEM:decay}
Let $d \ge 2$.
Assume that $a\in \mathbb{R}^d, |a|=1$, and the initial data $\phi\in L^2(\mathbb{R}^d)$ with 
$$
{\rm supp}\, \widehat{\phi} 
\subset 
\{ \xi \in \mathbb{R}^d : |a\cdot \xi|\le M\},
$$
then the estimate
\begin{equation}\label{e:decay1}
\|e^{it \Delta} \phi\|_{L^{\infty}_{x}( \mathbb{R}^d )} \lesssim M  t^{-\frac{d-1}2} \| \phi\|_{L_x^1(\mathbb{R}^d )}   
\end{equation}
holds.   
\end{lemma}

\begin{proof}
By the rotation invariance, we may assume that 
\[
{\rm supp}\, \widehat{\phi} 
\subset 
\{ \xi \in \mathbb{R}^d : |\xi_1|\le M\}
\]
where $\xi = (\xi_1, \xi_2, \cdots, \xi_d)$. Then we have
\[
\begin{split}
e^{it \Delta} \phi (x) & = c_d \int_{\R^d} \widehat \phi(\xi) e^{it |\xi|^2 + i x \cdot \xi} d \xi \\
& = c_d \int_{-M}^M \bigg( \int_{\R^{d-1}} \widehat \phi(\xi_1,\xi') e^{it |\xi'|^2 + i x' \cdot \xi'} d \xi' \bigg) e^{it |\xi_1|^2} e^{i x_1 \xi_1} d\xi_1,
\end{split}
\]
where $x = (x_1,x')$ and $\xi = (\xi_1,\xi')$. 
Then, by Minkowski's inequality and dispersive estimate for Schr\"odinger equation in $d-1$ dimensions, we have
\[
\begin{split}
\|e^{it \Delta} \phi (x)\|_{L^\infty_x (\R^d)} & \lesssim \bigg\| \int_{-M}^M \bigg( \int_{\R^{d-1}} \widehat \phi(\xi_1,\xi') e^{it |\xi'|^2 + i x' \cdot \xi'} d \xi' \bigg) e^{it |\xi_1|^2} e^{i x_1 \xi_1} d\xi_1 \bigg\|_{L^\infty_x} \\
& \lesssim \int_{-M}^M \bigg\|  \int_{\R^{d-1}} \widehat \phi(\xi_1,\xi') e^{it |\xi'|^2 + i x' \cdot \xi'} d \xi' \bigg\|_{L^\infty_{x'}} d\xi_1 \\
& \lesssim t^{-\frac{d-1}2} \int_{-M}^M \bigg\|  \int_{\R^{d-1}} \widehat \phi(\xi_1,\xi') e^{i x' \cdot \xi'} d \xi' \bigg\|_{L^1_{x'}} d\xi_1,\\
& = t^{-\frac{d-1}2} \int_{-M}^M \bigg\|  \int_{\R } \phi(x_1,x') e^{i x_1 \cdot \xi_1} d x_1 \bigg\|_{L^1_{x'}} d\xi_1\\
& \le t^{-\frac{d-1}2} 2M \sup_{\xi_1 \in \R} \bigg\|  \int_{\R } \phi(x_1,x') e^{i x_1 \cdot \xi_1} d x_1 \bigg\|_{L^1_{x'}}  \\
& \le  2Mt^{-\frac{d-1}2} \|\phi\|_{L^1_x (\R^d)}.
\end{split}
\]
\end{proof}

We say that the pair $(q,r)$ is sharp $\sigma$-admissible if
\[
2\le q,r\le \infty, \quad \frac1q = \sigma \bigg( \frac12 -\frac1r\bigg), \quad (q,r,\sigma) \neq (2,\infty,1).
\]

With the dispersive decay estimate \eqref{e:decay1},
we conclude the following Strichartz estimates according to Keel-Tao's machinery \cite{keel1998endpoint}.
\begin{theorem}[Strichartz estimate in the frequency-restricted setting]
\label{LEM:Stri}
Let $d\ge 2$ and $\phi$ be as in Lemma \ref{LEM:decay}.
Then, we have
\[
\| e^{it \Delta} \phi \|_{L^q(\R; L_x^r(\R^d))} \lesssim C(M) \|\phi\|_{L^2_x (\R^d)},
\]
where $(q,r)$ is sharp $\frac{d-1}2$-admissible. 
\end{theorem}
\begin{remark}
Strichartz estimate for the strip-restricted data can be compared with standard Strichartz estimate with dimension lower by one.  Analogous Strichartz estimates for the periodic case and the waveguide case and can be obtained respectively, in view of Bernstein inequality.  
\end{remark}
Moreover, we recall the inhomogeneous Strichartz estimates in \cite{foschi2005inhomogeneous}.
\begin{lemma}
\label{LEM:inhom}
Let $1 \le q, \tilde q \le \infty$ and $1 \le  r, \tilde r < \infty$.
We assume the following scaling condition
\[
\frac1q + \frac1{\tilde q} = \frac{d}{2}(1 - \frac1r - \frac1{\tilde r}).
\]
Then we have
\[
\bigg\| \int_0^t e^{i(t-t') \Delta} F(t') dt' \bigg\|_{L_t^q (\R; L^r_x (\R^d))} \lesssim \|F\|_{L_t^{\tilde q'} (\R; L^{\tilde r'}_x (\R^d))},
\]
where $\tilde q'$ and $\tilde r'$ are conjugate exponents of $\tilde q$ and $\tilde r$, respectively.
\end{lemma}

\subsection{Local well-posedness in the frequency-restricted setting}

In this subsection, we establish local well-posedness for a specific NLS model: the quintic NLS in the homogeneous Sobolev space \(\dot{H}^{\frac{1}{8}}(\mathbb{R}^2)\), assuming the initial data satisfies a strip-type frequency restriction. For other NLS models, analogous results can be obtained in a similar way.

We investigate the following quintic Schr\"odinger equation posed on $\R^2$:
\begin{align}\label{NLS}
\begin{cases}
i \partial_t u + \Delta u = |u|^4 u\\
u(0) = \phi 
\end{cases} (t,x) \in \R \times \R^2,
\end{align}
where the restricted initial data $\phi$ satisfies 
\begin{align}\label{rest}
{\rm supp}\, \widehat{\phi} 
\subset 
\{ \xi \in \mathbb{R}^2 : |a \cdot \xi|\le M\},
\end{align}
for a given $a \in \R^2$ such that $|a| = 1$.

We note that the critical space for \eqref{NLS} is $\dot H^{\frac12} (\R^2)$\footnote{In short, we say a NLS model is $\dot{H}^s$-critical if the $\dot{H}^s$-norm is invariant under the scaling symmetry. We refer to \cite{tao2006nonlinear} for the notion of criticality of NLS.}.
In general, it is expected that the initial value problem \eqref{NLS} is well-posedness in critical space or higher regularity spaces, that is, in $ H^s (\R^2)$ for $s \ge \frac12$; but ill-posed otherwise (see \cite{lindblad1995existence,CCT}). In what follows, however, we demonstrate that under the additional restriction on initial condition \eqref{rest}, \eqref{NLS} is well-posed in the supercritical regime.\footnote{To our best knowledge, this result is new, though the setting and the observation are quite natural.} We have the following theorem,

\begin{theorem}[Local well-posedness under frequency restriction]\label{thm:LWP-restricted-H14}
Let \(d = 2\), and consider the quintic NLS
\[
\begin{cases}
i \partial_t u + \Delta u = |u|^4 u, \\
u(0, x) = \phi(x),
\end{cases}
\]
on \(\mathbb{R}^2\). Suppose the initial data \(\phi \in \dot{H}^{\frac{1}{8}}(\mathbb{R}^2)\) satisfies the frequency restriction
\[
\operatorname{supp}(\widehat{\phi}) \subset \{ \xi \in \mathbb{R}^2 : |\xi_1| \le M \},
\]
for some fixed \(M > 0\). Then there exists \(T = T(\|\phi\|_{\dot{H}^{\frac{1}{8}}}) > 0\) such that the equation admits a unique solution \(u \in C([0,T]; \dot{H}^{\frac{1}{8}}(\mathbb{R}^2))\), depending continuously on the initial data.
\end{theorem}
The proof of Theorem \ref{thm:LWP-restricted-H14} is based on a standard contraction mapping argument in a suitable function space. We give a sketch of the proof as follows. 
\begin{proof}
Let $I=[0,T]$ and $T>0$ small enough such that,
\begin{equation}
    \||\nabla|^{\frac{1}{8}} e^{it\Delta} u_0\|_{L^8_t L^4_x(I\times \R^{2})}\leq \eta.
\end{equation}

Let us define the solution space:
\[
X_T := C_t([0,T]; \dot{H}^{\frac{1}{8}}_x) \cap L^8_t([0,T]; L^8_x).
\]
We equip this space with the norm:
\[
\|u\|_{X_T} := \|u\|_{L^8_t L^8_x}.
\]
We consider the set
\begin{equation}\label{eq:e  set}
\begin{aligned}
E:=\{&u:\ \|u\|_{L_{t}^{\infty}\dot H^{\frac{1}{8}}(I\times \R^{2})}
\leq 
2 \| \phi\|_{\dot H_{x}^{\frac{1}{8}} };\\& \| u\|_{L^8_t L^8_x(I\times \R^{2})}\leq 2\eta.\ \}
\end{aligned}
\end{equation}

According to its Duhamel formulation, let
\[
\Phi(u)(t) := e^{it\Delta} \phi - i \int_0^t e^{i(t - s)\Delta} (|u|^4 u)(s) ds.
\]

Then, by Sobolev inequality, the H\"older, the Bernstein\footnote{Essentially, the Bernstein in 1D is applied since the other direction is restricted.}, Theorem \ref{LEM:Stri} with $(q,r) = (8,4)$ and Lemma \ref{LEM:inhom} with $d = 2$, $(q,r) = (8,8)$, and $(\tilde q, \tilde r) = (\frac83, \frac83)$, one can show $\Phi(u)$ is from $E$ to $E$ via
\begin{equation}
    \|\Phi(u)\|_{L_{t}^{\infty}\dot H^{\frac{1}{8}}(I\times \R^{2})}
\leq \| \phi\|_{\dot H_{x}^{\frac{1}{8}} }+ \| u\|_{\dot H_{x}^{\frac{1}{8}} } \eta^4 \leq 
2 \| \phi\|_{\dot H_{x}^{\frac{1}{8}} },
\end{equation}
and
\begin{equation}
   \| \Phi(u)\|_{L^8_{t,x}(I\times \R^{2})}\leq \eta+\eta^5 \leq 2\eta
\end{equation}
since
\begin{align}
\begin{split}
\|u\|_{L^{8}_{t,x} (\R \times \R^2)} & \le \| e^{it \Delta} \phi \|_{L^8(\R; L_x^8(\R^2))} + \bigg\| \int_0^t e^{i(t-t') \Delta}  |u|^4 u (t') dt' \bigg\|_{L_t^8 (\R; L^8_x (\R^2))}\\
& \lesssim_M \| |\nabla|^{\frac18} e^{it \Delta} \phi \|_{L^8(\R; L_x^4(\R^2))} + \big\|  |u|^4 u  \big\|_{L_t^\frac85 (\R; L^\frac85_x (\R^2))}\\
& \lesssim_M \eta + \| u\|_{L^8_{t,x} (\R\times \R^2)}^5.
\end{split}
\end{align}
Moreover, $\Phi(u)$ is a contraction mapping,
\begin{equation}
   \|\Phi(u)-\Phi(v)\|_{L^8_{t,x}(I\times \R^{2})}\leq \|u-v\|_{L^8_{t,x}(I\times \R^{2})} C\eta^{4}.
\end{equation}

Thus, for small enough \(T\), the map \(\Phi\) is a contraction on a ball in \(X_T\), yielding a unique local solution by Banach fixed point theorem.
\end{proof}

Furthermore, this argument can be extended to general NLS models in a standard way with suitable modifications. We proceed to investigate the general case as below.
\begin{align}\label{general NLS}
\begin{cases}
i \partial_t u + \Delta u = |u|^p u\\
u(0) = \phi 
\end{cases} (t,x) \in \R \times \R^d,
\end{align}
where the restricted initial data $\phi$ satisfies 
\begin{align}\label{general rest}
{\rm supp}\, \widehat{\phi} 
\subset 
\{ \xi \in \mathbb{R}^d : |a \cdot \xi|\le M\},
\end{align}
for a given $a \in \R^d$ such that $|a| = 1$.

We now formulate the well-posedness theory for \eqref{general NLS} under a localized frequency assumption \eqref{general rest} on the initial data. This setting allows us to establish local well-posedness in certain supercritical sense, as stated in the following theorem.

\begin{theorem}
\label{Thm:well-posedness}
Consider \eqref{general NLS}. Let $d\geq 2$, $s= \frac{p(d-1)(d+2)-4(d+1)}{2p(d+2)}$ and the nonlinear exponent $p=\frac{4}{d-2s_c}$, where the critical regularity is $s_{c}=\frac{d}{2} - \frac{2}{p}\geq 0$. Assume that $\phi\in \dot H_{x}^{s}( \R^{d})$ satisfies the restriction condition \eqref{general rest}. There exists a constant $\eta_{0}=\eta_{0}(d)>0$ such that if $0<\eta\leq \eta_{0}$ and $I$ is a compact interval containing zero such that
\begin{equation}\label{small condition}
\| |\nabla|^{s} e^{it \Delta} \phi \|_{L_{t,x}^q (I \times \R^d)} \leq \eta,  
\end{equation}
then the Cauchy problem \eqref{general NLS} is locally well-posed in the supercritical sense since $s<s_c$.
\end{theorem}
Since the proof follows as the 2D quintic NLS case (Theorem \ref{thm:LWP-restricted-H14}), we omit it. Here are a few more remarks on this research line.
\begin{remark}[On the definiteness of supercriticality]
There exists initial data $\phi$ satisfying the assumptions in \eqref{general rest} with $\phi \in \dot H^s$ yet $\phi \notin \dot H^{s_c}$,. A constructive example can be given as follows: 

Fix $a = (0,\ldots,0,1)$ and choose an intermediate regularity exponent $s_0 \in (s, s_c)$. Define the Fourier transform
\begin{equation*}
    \hat{\phi}(\xi) = \begin{cases}
        (1+|\xi_1|)^{-s_0-\frac12}, & \text{if } |\xi_k| \leq 1\ \text{for all } 2 \leq k \leq d, \\
        0, & \text{otherwise}.
    \end{cases}
\end{equation*}
This construction ensures $\|\phi\|_{\dot H^s} < \infty$ but $\|\phi\|_{\dot H^{s_c}} = \infty$ due to the anisotropic decay rate $s_0$ in the $\xi_1$-direction combined with compact support in other frequency variables.

\end{remark}

\begin{remark}[On the regularity improvement]
A direct computation shows that the inequality
$\frac{d}{2}-\frac{2}{p}>\frac{p(d-1)(d+2)-4(d+1)}{2p(d+2)}$ 
holds for all \( d \geq 2 \) and \( p > 0 \). This indicates that the regularity required for well-posedness is lowered from \( \frac{d}{2} - \frac{2}{p} \) (standard critical regularity) to \( \frac{p(d-1)(d+2)-4(d+1)}{2p(d+2)} \), which in turn implies that well-posedness can be achieved even for certain supercritical cases of the NLS.
\end{remark}

\begin{remark}[On a natural generalization]
One can generalize Theorem \ref{Thm:well-posedness} in a natural way: consider $d$-dimensional NLS \eqref{general NLS} with multiple directions ($m$-dimension) restricted in frequency where $d\geq3, m\geq 2,d>m$. Analogous result can be obtained. In fact, the more directions are restricted, the better regularity improvement is expected.
\end{remark}

\begin{remark}[On small data global well-posedness]
Via standard methods (see \cite{tao2006nonlinear,dodson2019defocusing,killip2013nonlinear}), one can also obtain small data global well-posedness results for supercritical NLS as concerned in Theorem \ref{Thm:well-posedness} so we omit it. The large data long time dynamics remains open and requires other ingredients.   
\end{remark}
\begin{remark}[On the waveguide case]
Such results (well-posedness results for supercritical NLS) can be easily extended to the waveguide case with suitable modifications. As in \cite{tzvetkov2012small}, one can use Theorem \ref{LEM:Stri} (Euclidean Strichartz estimates) as a blackbox to obtain Strichartz estimates in the waveguide setting, which lead well-posedness results for supercritical NLS on waveguide manifolds with regularity improvements.   
\end{remark}
\begin{remark}[On the wave equations]
Such results are also expected to be extended to the nonlinear wave equations (even other nonlinear dispersive equations) with suitable modifications. The road map is clear: dispersive estimate in lower dimensional space implies Strichartz estimate in lower dimensional space, which gives the well-posedness theory for supercritical models. We refer to \cite{keel1998endpoint} (see Theorem 1.2 and Theorem 10.1).
\end{remark}

\begin{remark}[On the optimality (well/ill-posedness)]
It is natural to ask if the well-posedness result Theorem \ref{Thm:well-posedness} is optimal in terms of the Sobolev regularity? In other others, can one prove ill-posedness for any strip-restricted data with lower Sobolev regularity? 

In fact, it is not clear for this moment. We still use the 2D quintic NLS as an example as in Theorem \ref{thm:LWP-restricted-H14}. We point out that, following the standard scheme \cite{CCT} (norm inflation), one can show ill-posedness for \eqref{NLS} in $H^{s}$ for $s<0$ by reducing one dimension. There is still an obvious gap between $0$ and $\frac{1}{8}$.    
We leave it for future studies.
\end{remark}

\section{Summary and Open Problems}\label{7}
In this section, we make a summary and present some further remarks on the research line of restricted-type Strichartz estimates in different settings and the applications.

\subsection{Summary of the shell-type Strichartz estimates}

We begin with a concise summary of the shell-type Strichartz estimates established in this work.

1. In general, the analogous shell-type Strichartz estimates for $\mathbb{R}^m \times \mathbb{T}^n$ can be obtained compared to the periodic case (see \cite{MR4665720}). Except for the case $\mathbb{R}^2 \times \mathbb{T}$, there is $\epsilon$-derivative loss and the estimates are local-in-time. (See Section \ref{5}.)

2. In particular, on one hand, the shell-type Strichartz estimates for $\mathbb{R} \times \mathbb{T}^{d-1}$  must have a derivative loss (there are counter-examples to justify this point); on the other hand, the shell-type Strichartz estimate for $\mathbb{R}^2 \times \mathbb{T}$ is global-in-time and there is no derivative loss. One can compare this case with the $\mathbb{R}^2 \times \mathbb{T}$ case and the $\mathbb{T}^3$ case. (See Section \ref{4}.)

3. Finally, the shell-type Strichartz estimates for the Euclidean case are global-in-time, and there is no derivative loss. (See Section \ref{3}.)

For convenience, we refer to the following table for a brief summary based on the current paper and \cite{MR4665720}.

\begin{table}[htbp]
\centering
\caption{Summary of the shell-type Strichartz estimates}
\label{tab:example}
\begin{tabular}{|c|c|c|c|} 
\toprule
\textbf{shell-type Strichartz} & \textbf{Euclidean case} & \textbf{Waveguide case} & \textbf{Periodic case} \\
\midrule
Derivative loss & no loss & no loss for $\mathbb{R}^2\times \mathbb{T}$ and $\epsilon$-loss for $\mathbb{R}\times \mathbb{T}^d$ & $\epsilon$-loss\\
Globality & global-in-time & global-in-time for $\mathbb{R}^2\times \mathbb{T}$ & local-in-time \\
\bottomrule
\end{tabular}
\end{table}

It remains to study the case $\mathbb{R}^m\times \mathbb{T}^n$ ($m\geq 2,n\geq 1$\footnote{Expect for the case $m=2,n=1$, which is known due to this paper.}) for the globality and the derivative loss issue. We leave them for interested readers.

As direct PDE applications (see Section \ref{6}), we obtain the well-posedness results for the partially periodic Zakharov system following \cite{MR4665720} with proper modifications.

Lastly, as a comparison, we note that, for the standard Strichartz estimates for Schr\"odinger equations, one can make a similar summary as below due to existing results.\footnote{One can also compare Strichartz estimates in different settings from the aspect of the range of admissible pairs. (The range in the Euclidean setting is larger than its analogue in the periodic setting.) Moreover, one can also take two more cases, i.e. Schr\"odinger equations under a (partial) harmonic confinement, into considerations (see \cite{antonelli2015scattering,jao2016energy} and the references therein for more details). We leave them for interested readers. See also \cite{wang2013periodic} for a sharp Strichartz estimate for the hyperbolic case.} 

\begin{table}[htbp]
\centering
\caption{Summary of Strichartz estimates}
\label{tab:example}
\begin{tabular}{|c|c|c|c|} 
\toprule
\textbf{Strichartz estimate} & \textbf{Euclidean case} & \textbf{Waveguide case} & \textbf{Periodic case} \\
\midrule
Derivative loss & no loss (see \cite{keel1998endpoint}) & no loss for $\mathbb{R}\times \mathbb{T}$ (see \cite{takaoka20012d})& $\epsilon$-loss (see \cite{BD15})\\
Globality &  global-in-time (see \cite{keel1998endpoint}) & (weakly) global-in-time (see \cite{Barron14}) & local-in-time (see \cite{BD15})\\
\bottomrule
\end{tabular}
\end{table}
\begin{remark}
    By ``(weakly) global-in-time'', we mean that the Strichartz estimate in the waveguide setting is global-in-time but the time integrability is not as good as the Euclidean analogue (they are of type $l^qL^p_{t,x}$ ($q>p$)) where $l^q$ represents the time integrability. We refer to \cite{Barron14} for more details.
\end{remark}

\subsection*{Comparison between Strip-Type and Shell-Type Restrictions}
Strip-type Strichartz estimates focus on frequency localization along a single direction, capturing wave packets concentrated along narrow slabs, and are particularly useful for studying solutions with partial dispersion. In contrast, shell-type Strichartz estimates focus on frequencies constrained to a thin spherical shell, preserving radial symmetry while allowing angular freedom, and are crucial for understanding solutions with fixed energy levels. Technically, strip-type estimates rely more on anisotropic inequalities and wave packet decompositions, while shell-type estimates often invoke spherical harmonics and restriction theorems. Both of the two types of restrictions are natural choices in view of applications.

As shown in previous sections, shell-type Strichartz estimates and strip-type Strichartz estimates share an important feature: \textit{they can be compared with classical Strichartz estimates with dimension lower by one}. One difference is that strip-type Strichartz estimates are of \textit{full Strichartz range} (see Theorem \ref{LEM:Stri}), which is not proven yet for shell-type Strichartz estimates. We leave it for interested readers. In fact, if shell-type Strichartz estimates with full Strichartz range hold, Section \ref{new} can be easily\textit{ recovered} for the shell-restricted case. 

It is also worth emphasizing that both types of restrictions inherently require dimension $d \geq 2$. In one dimension, neither the concept of a strip nor a shell is meaningful: there are no transverse directions to define a strip, and the ``sphere'' $S^0$ consists of only two points, precluding any notion of angular dispersion. Thus, the phenomena captured by strip-type or shell-type estimates fundamentally rely on multi-dimensional spatial structures.
\subsection{Open questions of this research line} 

Finally, we remark on open problems and future directions.
\subsubsection{On shell-type Strichartz estimates and the Zakharov system}
On the \textbf{analysis} level, it is interesting to investigate whether the shell-type Strichartz estimates have derivative loss on the higher dimensional waveguide manifolds $\mathbb{R}^m \times \mathbb{T}^n$ ($d=m+n \geq 4$).\footnote{The estimates in 1D, 2D and 3D are clear. For cases $\mathbb{R}^m \times \mathbb{T}^{n}$ ($m\geq 2,m+n \geq 4$), we do not know if the derivative loss appears or not.} We also note that, even for the standard Strichartz estimates in the waveguide setting, this problem is interesting and not well studied yet (see Table \ref{tab:example}).\footnote{According to existing results, the only known case is Strichartz estimate on $\mathbb{R}\times \mathbb{T}$ (There is no derivative loss at the endpoint. See \cite{takaoka20012d}).}  In general, $\epsilon$-derivative loss will appear at the endpoint exponent by applying the decoupling method (see \cite{BD15}); when the exponent is away from the endpoint, it is possible to remove the derivative loss (see \cite{KV16} and \cite{Barron14} for the tori case and the waveguide case, respectively). For the endpoint case, other methods/ingredients are required (rather than the decoupling method) if one wants to establish no-derivative-loss estimates in the waveguide setting. Besides the derivative loss issue, it is also interesting to investigate whether the estimates are global-in-time. We leave these interesting problems for future studies.

On the \textbf{PDE} level, it is interesting to further investigate the well-posedness theory (such as the regularity threshold for well-posedness/ill-posedness) and even the long time behavior (such as scattering) for Zakharov systems on waveguide manifolds. Since there are some dispersive effects due to the Euclidean components, scattering behavior is also expected\footnote{Since the Zakharov system is a coupled system of NLS and nonlinear wave equation (NLW), studying the scattering for NLW on waveguide manifolds would be a first step.}, which is qualitatively distinct from the periodic case. Zakharov systems (and some other nonlinear dispersive models) on waveguide manifolds are not well studied yet, and there are many related interesting problems.\footnote{A general direction is to generalize ``NLS on waveguides manifolds" to ``nonlinear dispersive equations on waveguide manifolds".} Furthermore, besides the study of the (partially) periodic Zakharov system, one can investigate other applications of the shell-type Strichartz estimates in different settings.\footnote{We have presented one application in the Euclidean setting, i.e. proving well-posedness for supercritical NLS with restricted initial data. See Section \ref{new}.}
\subsubsection{On the strip-type Strichartz estimates and the supercritical NLS}
It is natural to ask: can the results for supercritical NLS in Section \ref{new} be further extended? It is possible to study the \textit{large data global well-posedness and scattering for supercritical NLS} with strip-restricted initial data. It is expected to obtain the large data global well-posedness and scattering for supercritical NLS with strip-restricted initial data, combined with classical I-method, random data theory for NLS or other methods (see \cite{tao2006nonlinear}). It is another interesting topic and we study it in our next paper \cite{next}. Moreover, one can also investigate the \textit{optimal local well-posedness} problem as stated in the end of Section \ref{new}.

Finally, to summarize the open problems discussed in this subsection, we refer to Table \ref{tab:open_problems} as below.

\begin{table}[h]
\centering
\caption{Key open problems}
\label{tab:open_problems}
\begin{tabular}{|l|l|}
\hline
\textbf{Problems} & \textbf{Related Materials} \\ 
\hline
Sharp (log-type) derivative loss for $\mathbb{R}\times\mathbb{T}^{d-1}$ and $\mathbb{T}^{d}$  & Sections \ref{3}, \ref{4}, \ref{5} and \cite{MR4665720} \\
Derivative loss issue for Strichartz estimates on $\mathbb{R}^m\times\mathbb{T}^n$ ($m\geq2,m+n\geq4$) & Sections \ref{3}, \ref{4}, \ref{5} \\
Full Strichartz range for shell-type Strichartz estimates & Sections \ref{3}, \ref{4}, \ref{5} \\
Optimal well-posedness theory for Zakharov systems on waveguides & Section \ref{6} \\
Scattering theory for Zakharov systems \& NLW on waveguides & Section \ref{6} \\
Optimal well-posedness theory for supercritical NLS with strip-restricted data & Section \ref{new} \\
Large data GWP and scattering for supercritical NLS with strip-restricted data & Section \ref{new} and \cite{next} \\
Understanding: random data theory v.s. strip-restriction condition & Section \ref{new} and \cite{deng2022random,deng2024invariant,bringmann2024invariant} \\
\hline
\end{tabular}
\end{table}

\section{Numerical Experiments and Verification of Estimates}\label{9}
In this section, we utilize numerical verifications to discuss shell-type Strichartz estimates on waveguide manifolds and tori, respectively. The results are consistent with their theoretic analogues established in this paper.

We consider the three-dimensional case as an example (see Table \ref{tab:3d}). It is surely possible to investigate numerical verifications for other dimensions, such as the 2D case, which is easier and the corresponding theoretic results are also clear. (For the $\mathbb{R}^2$ case, there is no derivative loss; for both the $\mathbb{R}\times \mathbb{T}$ case and the $\mathbb{T}^2$ case, there is an inevitable derivative loss.) In general, when the dimension grows, the numerical verifications will become much more complicated, and we leave the higher dimensional cases for interested scholars.

We validate the shell-type Strichartz estimates for the linear Schr\"odinger equation on three geometries\footnote{Since the Euclidean case is well-known, we omit it.}:
\begin{equation}
    \|e^{it\Delta}u_0\|_{L^4_{t,x}([0,1] \times M)} \leq C(N)\|u_0\|_{L^2}, \quad \text{supp}\,\hat{u}_0 \subset \{k \mid |k| \in [N-1, N+1]\},
\end{equation}
where \(M \in \{\mathbb{R}^2 \times \mathbb{T}, \mathbb{R} \times \mathbb{T}^2, \mathbb{T}^3\}\). Our simulations highlight the role of partial periodicity in derivative loss phenomena. We refer to Theorem \ref{mainthm on R2T}, Theorem \ref{counterexample on RTd-1}, Theorem \ref{mainthm on RmTn} and \cite{MR4665720} for the theoretic results.

\subsection{Implementation}
For each geometry, we use a pseudo-spectral method with the following configurations:
\begin{itemize}
    \item \textbf{Waveguide (\(\mathbb{R}^2 \times \mathbb{T}\))}: Spatial domain \([-8\pi, 8\pi]^2 \times [0, 2\pi]\) with \(256^2 \times 32\) grid points.
    \item \textbf{Partial Periodic (\(\mathbb{R} \times \mathbb{T}^2\))}: Spatial domain \([-8\pi, 8\pi] \times [0, 2\pi]^2\) with \(256 \times 32^2\) grid points.
    \item \textbf{Periodic (\(\mathbb{T}^3\))}: Uniform \(64^3\) grid on \([0, 2\pi]^3\).
\end{itemize}
Initial data \(u_0\) is generated with Fourier support restricted to shells of width 2 centered at \(N\).

To quantify the derivative loss, we perform linear regression on the log-log plot to estimate the scaling exponent $\alpha$ in $C(N) \sim N^\alpha$. The numerical results confirm our theoretical predictions.
\subsection{Results and Analysis}
\begin{table}[h]
    \centering
    \caption{Scaling of \(L^4_{t,x}\) norms with frequency \(N\)}
    \begin{tabular}{lcc}
        \toprule
        \textbf{Geometry} & \textbf{Observed Scaling \(C(N)\)} & \textbf{Theoretical Prediction} \\
        \midrule
        \(\mathbb{R}^2 \times \mathbb{T}\) & \(O(1)\) & No derivative loss (Theorem \ref{mainthm on R2T}) \\
        \(\mathbb{R} \times \mathbb{T}^2\) & \(O(N^{0.18})\) & Mild \(\epsilon\)-loss (Theorem \ref{counterexample on RTd-1}) \\
        \(\mathbb{T}^3\) & \(O(N^{0.30})\) & \(\epsilon\)-loss (\cite{MR4665720}) \\
        \bottomrule
    \end{tabular}
    \label{tab:scaling}
\end{table}

Key observations:
\begin{itemize}
    \item \textbf{No Derivative Loss}: For \(\mathbb{R}^2 \times \mathbb{T}\), the \(L^4\)-spacetime norm is bounded uniformly in \(N\), confirming Theorem \ref{mainthm on R2T}.
    \item \textbf{Intermediate Behavior}: \(\mathbb{R} \times \mathbb{T}^2\) exhibits a weaker scaling (\(N^{0.18}\)) compared to \(\mathbb{T}^3\) (\(N^{0.30}\)), reflecting the suppression of energy transfer due to the Euclidean component.
    \item \textbf{Consistency with Theory}: All cases align with theoretical predictions.
\end{itemize}

\begin{figure}[h]
    \centering
    \includegraphics[width=0.8\textwidth]{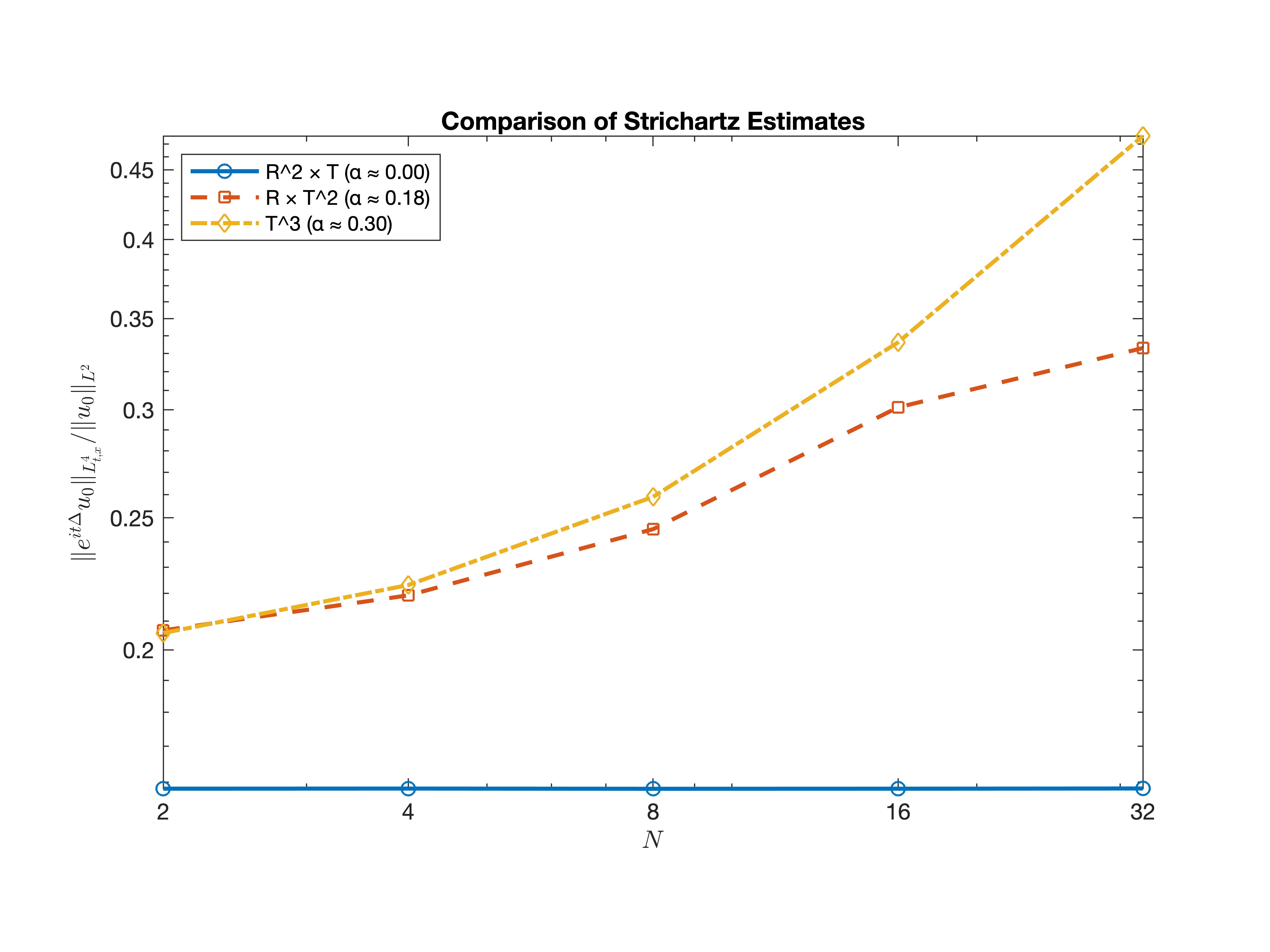}
    \caption{Log-log plot of \(\|e^{it\Delta}u_0\|_{L^4}/\|u_0\|_{L^2}\) vs. \(N\) for \(N \in [4, 32]\).}
    \label{fig:scaling_plot}
\end{figure}

\subsection{Discussions}
The numerical results reveal an interesting hierarchy in derivative loss phenomena:
\begin{itemize}
\item \textbf{Strong dispersion}: The $\mathbb{R}^2\times\mathbb{T}$ case benefits from two Euclidean directions providing strong dispersion that suppresses energy transfer to high frequencies, preventing derivative loss.

\item \textbf{Intermediate case}: For $\mathbb{R}\times\mathbb{T}^2$, the single Euclidean direction leads to partial dispersion, resulting in a subcritical derivative loss ($\alpha=0.18$) that is significantly smaller than the periodic case.

\item \textbf{Periodic confinement}: The $\mathbb{T}^3$ case shows the strongest derivative loss ($\alpha=0.30$) due to complete lack of dispersion in all directions.
\end{itemize}

\subsection{Physical Interpretation}
The observed scaling hierarchy can be interpreted via \textit{dispersion strength}:
\begin{itemize}
\item \textbf{Strong dispersion} ($\mathbb{R}^2\times\mathbb{T}$): Two Euclidean dimensions provide sufficient dispersion to suppress energy cascade to high modes, yielding $C(N)\sim 1$.
\item \textbf{Weak confinement} ($\mathbb{R}\times\mathbb{T}^2$): Single Euclidean dimension leads to partial localization, causing mild derivative loss ($\alpha=0.18$).
\item \textbf{Full confinement} ($\mathbb{T}^3$): Complete absence of dispersion results in strong derivative loss ($\alpha=0.30$).
\end{itemize}

These observations suggest a general principle: the derivative loss in Strichartz estimates is controlled by the \textit{dimension deficit} between the Euclidean and periodic components. This provides guidance for designing waveguide structures in photonic applications where dispersion properties are crucial. Open questions include what the optimal derivative losses are for both the \(\mathbb{R} \times \mathbb{T}^2\) case and the \( \mathbb{T}^3\) case respectively, and if analogous behavior holds for higher dimensions (e.g., \(\mathbb{R}^m \times \mathbb{T}^n\) with \(m+n \geq 4\)). 

\subsection{Remarks on the classical Strichartz estimates}
Besides the shell-type estimates, one can also investigate the classical Strichartz estimates via numerical verifications with modifications. One may still consider the three-dimensional case as an example. It is fine to investigate numerical verifications for other dimensions, such as the $L^4$-Strichartz estimates in 2D (considering the $\mathbb{R} \times \mathbb{T}, \mathbb{T}^2$ cases), which would be easier and the corresponding theoretic results are more clear.

One can validate the local-in-time Strichartz estimates for initial data with Fourier support in \textbf{balls}:
\begin{equation}
    \|e^{it\Delta}u_0\|_{L^{10/3}_{t,x}([0,1] \times M)} \leq C(N)\|u_0\|_{L^2}, \quad \text{supp}\,\hat{u}_0 \subset \{k \mid |k| \leq N\},
\end{equation}
where \(M \in \{\mathbb{R}^2 \times \mathbb{T}, \mathbb{R} \times \mathbb{T}^2, \mathbb{T}^3\}\). We refer to \cite{BD15,KV16,Barron14} for the existing  theoretic results, i.e. Strichartz estimates in periodic setting and waveguide setting respectively. \footnote{The 3D Strichartz estimates at the endpoint $\frac{10}{3}$ all have derivative loss and it is not proven yet whether the derivative loss is inevitable. The only clear case in the waveguide setting is the $L^4$ (endpoint) Strichartz estimate on $\mathbb{R}\times \mathbb{T}$ (There is no derivative loss at the endpoint. See \cite{takaoka20012d}).}

We leave it for interested readers since the process is similar.\vspace{3mm}

\subsection*{Data Availability Statement:} The MATLAB codes used to generate the numerical results in Section \ref{9} are upon request from the corresponding author. The implementation includes:
\begin{itemize}
\item Spectral discretization of Laplacian on mixed geometries;
\item Time-splitting algorithm for linear Schr\"odinger equation;
\item Scripts for automated parameter studies.\vspace{3mm}
\end{itemize}

\bibliographystyle{abbrv}

\bibliography{BG}
\end{document}